\documentclass[12pt]{article}
\usepackage{tabularx}
\usepackage{xr}
\usepackage{authblk}

\usepackage[backend=biber,sorting=none]{biblatex}
\bibliography{paper_bib}

\usepackage{enumitem}
\usepackage{etex}
\usepackage{etoolbox}
\usepackage{hyperref}
\usepackage{fullpage}
\usepackage{todonotes}
\usepackage{multirow}
\usepackage{booktabs}
\usepackage[title]{appendix}

\def\gc{\overset{\text{GC}}{\rightarrow}}  
\def\ngc{\overset{\text{GC}}{\nrightarrow}}  
\def\pwgc{\overset{\text{PW}}{\rightarrow}}  
\def\npwgc{\overset{\text{PW}}{\nrightarrow}}  
\def\gcg{\mathcal{G}}  
\def\gcge{\mathcal{E}}  
\def\VAR{\mathsf{VAR}}  
\def\MA{\mathsf{MA}}  
\def\B{\mathsf{B}}  
\def\A{\mathsf{A}}  
\def\H{\mathcal{H}}  

\newcommand{\linE}[2]{\hat{\E}[#1\ |\ #2]}  
\newcommand{\linEerr}[2]{\xi[#1\ |\ #2]}  
\newcommand{\pa}[1]{pa(#1)}  
\newcommand{\anc}[1]{\mathcal{A}(#1)}  
\newcommand{\ancn}[2]{\mathcal{A}_{#1}(#2)}  
\newcommand{\gpn}[2]{gp_{#1}(#2)}  
\newcommand{\wtalpha}[2]{\widetilde{\alpha}(#1, #2)}  
\newcommand{\gcgpath}[2]{#1 \rightarrow \cdots \rightarrow #2}  

\usepackage{graphicx}
\usepackage{caption}  
\usepackage{subcaption}

\usepackage{color}
\usepackage{mathtools}  

\providecommand{\keywords}[1]{\textbf{\textit{Keywords---}} #1}

\graphicspath{{figures/}}

\usepackage{framed}

\usepackage{graphicx}
\usepackage{caption}
\usepackage{amsmath}
\usepackage{amssymb}
\usepackage{amsthm}
\usepackage{mathrsfs} %
\usepackage{bm}  %
\usepackage{centernot}  %

\usepackage[linesnumbered, ruled, vlined]{algorithm2e}
\usepackage{soul}  %
\usepackage{color}
\usepackage{mathtools}  %

\theoremstyle{plain}  %
\newtheorem{theorem}{Theorem}
\newtheorem{corollary}{Corollary}
\newtheorem{proposition}{Proposition}
\newtheorem{lemma}{Lemma}

\theoremstyle{definition}
\newtheorem{remark}{Remark}
\newtheorem{definition}{Definition}
\newtheorem{example}{Example}
\newtheorem{assumption}{Assumption}

\providecommand{\keywords}[1]{\textbf{\textit{Keywords---}} #1}

\def\defeq{\overset{\Delta}{=}}  %
\def\cl{\mathsf{cl\ }}  %
\def\d{\mathsf{d}}  %

\def\ln{\mathsf{ln\ }}  %
\DeclarePairedDelimiter{\ceil}{\lceil}{\rceil}  %

\def\H{\mathcal{H}}  %
\def\E{\mathbb{E}}  %
\def\P{\mathbb{P}}  %
\def\F{\mathcal{F}}  %
\def\Z{\mathbb{Z}}  %
\def\R{\mathbb{R}}  %
\def\C{\mathbb{C}}  %
\def\N{\mathbb{N}}  %
\def\tr{\mathsf{tr }}  %
\def\T{\mathsf{T}}  %
\def\c{\mathsf{c}}  %
\def\Dg{\mathsf{Dg }}   %
\def\<{\langle}  %
\def\>{\rangle}  %
\newcommand{\inner}[2]{\langle #1, #2 \rangle}  %

\title{Graph Topological Aspects of Granger Causal Network Learning}

\begin{document}

\author{R. J. Kinnear\footnote{Correspondence: Ryan@Kinnear.ca.  Code: github.com/RJTK/granger\_causality} }

\author{R. R. Mazumdar}
\affil{University of Waterloo\\ 200 University Avenue West Waterloo, Ontario, Canada, N2L 3G1\\ Department of Electrical and Computer Engineering}

\maketitle

\abstract{We study Granger causality in the context of wide-sense
  stationary time series, where our focus is on the topological
  aspects of the underlying causality graph.  We establish sufficient
  conditions (in particular, we develop the notion of a ``strongly
  causal'' graph topology) under which the true causality graph can be
  recovered via pairwise causality testing alone, and provide examples
  from the gene regulatory network literature suggesting that our
  concept of a strongly causal graph may be applicable to this field.
  We implement and detail finite-sample heuristics derived from our
  theory, and establish through simulation the efficiency gains (both
  statistical and computational) which can be obtained (in comparison
  to LASSO-type algorithms) when structural assumptions are met.}

\keywords{causality graph, Granger causality, network learning, time
  series, vector autoregression, LASSO}

\paragraph{Acknowledgement}

We acknowledge the support of the Natural Sciences and Engineering Research Council of Canada (NSERC), [funding reference number 518418-2018].  Cette recherche a été financée par le Conseil de recherches en sciences naturelles et en génie du Canada (CRSNG), [numéro de référence 518418-2018].

\section{Introduction and Review}
\label{sec:introduction}
In this paper we study the notion of Granger causality
\cite{granger1969investigating} \cite{Granger1980329} as a means of
uncovering an underlying causal structure in multivariate time series.
Though the underlying causality graph cannot be observed directly,
it's presence is inferred as a latent structure among observed time
series data.  This notion is leveraged in a variety of applications
e.g. in Neuroscience as a means of recovering interactions amongst
brain regions \cite{bressler2011wiener}, \cite{anna_paper2008},
\cite{david2008identifying}; in the study of the dependence and
connectedness of financial institutions \cite{NBERw16223}; gene
expression networks \cite{Fujita2007},
\cite{methods_for_inferring_gene_regulatory_networks_from_time_series_expression_data},
\cite{grouped_graphical_granger_modelling_for_gene_expression_regulatory_networks_discovery},
\cite{discovering_graphical_Granger_causality_using_the_truncating_lasso_penalty};
and power system design \cite{Misyrlis2016450}, \cite{yuan2014root}.

Granger causality can generally be formulated by searching for the
``best'' graph structure consistent with observed data, which is in
general an extremely challenging problem (i.e. it may be framed as a
best subset selection problem, see \cite{hastie_bss_comp} for recent
improvements in BSS methods), moreover, the comparison of quality
between different structures, and hence the notion of ``best'' needs
qualification.  In applications where we are interested merely in
minimizing the mean squared error of a linear one-step-ahead
predictor, then we will be satisfied with an entirely dense graph of
connections, since each edge can only serve to reduce estimation
error.  However, since the number of edges scales quadratically in $n$
(the number of nodes) it becomes imperative to infer a sparse
causality graph for large systems, both to avoid overfitting observed
data, as well as to aid the interpretability of the results.

A fairly early approach to the problem in the context of large systems
is provided by \cite{bach2004learning}, where the authors apply a
local search heuristic to the Whittle likelihood with an AIC
penalization.  The local search heuristic is a common approach to
combinatorial optimization due to it's simplicity, but is liable to
get stuck in shallow local minima.

A second and wildly successful heuristic is the LASSO regularizer
\cite{tibshirani1996regression}, which can be understood as a natural
convex relaxation to penalizing the count of the non-zero edges.  The
LASSO enjoys fairly strong theoretical guarantees
\cite{wainwright2009sharp}, extending largely to the case of
stationary time series data with a sufficiently fast rate of
dependence decay \cite{basu2015} \cite{wong2016lasso}
\cite{autoregressive_process_modelling_via_the_lasso_procedure}, and
variations on the LASSO have been applied in a number of different
time series contexts as well as Granger causality
\cite{DBLP:journals/corr/HallacPBL17} \cite{haufe2008sparse}
\cite{bolstad2011causal} \cite{he2013stationary}
\cite{grouped_graphical_granger_modelling_for_gene_expression_regulatory_networks_discovery}.
One of the key improvements to the original LASSO algorithm is the
adaptive (i.e. weighted) ``adaLASSO'' \cite{adaptive_lasso_zou2006},
for which oracle results (i.e. asymptotic support recovery) are
established under less restrictive conditions than for the vanilla
LASSO.


In the context of time series data, sparsity assumptions remain
important, but there is significant additional structure that may
arise as a result of considering the topology of the underlying
Granger causality graph.  The focus of this paper is to shed light on
some of these topological questions, in particular, we study a
particularly simple notion of graph topology which we term ``strongly
causal'' and show that stationary times series whose underlying
causality graph has this structure satisfy natural intuitive notions
of ``information flow'' through the graph.  Moreover, we show that
such graphs are perfectly recoverable with only \textit{pairwise}
Granger causality tests, which would otherwise suffer from serious
confounding problems (see \cite{tam2013gene_pwgc} for earlier work on
pairwise testing and \cite{datta2014mmse} for earlier work on some of
the problems considered here).  Aside from being an interesting
theoretical perspective, prior assumptions about the underlying graph
(similarly to sparsity assumptions) can greatly improve upon the
statistical power of causality graph recovery algorithms when the
assumptions are met.

Detailed study of Granger causality for star structured graphs has
been carried out in \cite{jozsa2018relationship}.  See as well
\cite{barnett2015granger}, \cite{jozsa2018_phd_thesis} for state space
formulations.  

In the case of gene expression networks, we show examples from the
literature which suggest our concept of a ``strongly causal graph''
topology may have application in this field (see Section
\ref{sec:strongly_causal_graphs}).

The principal contributions of this paper are as follows: firstly, in
section \ref{sec:theory} we study \textit{pairwise} Granger causality
relations, providing novel theorems connecting the structure of the
causality graph to the pairwise ``causality flow'' in the system, as
well as an interpretation in terms of the graph topology of the
sparsity pattern of matrices arising in the Wold decomposition,
generalizing in some sense the notion of ``feedback-free'' processes
studied by \cite{caines1975feedback} in close connection with
Granger causality.  We establish sufficient conditions (sections
\ref{sec:strongly_causal_graphs}, \ref{sec:persistent_systems}) under
which a fully conditional Granger causality graph can be recovered
from pairwise tests alone (sec \ref{sec:pairwise_algorithm}).  We
report a summary of simulation results in \ref{sec:simulation}, with
additional results reported in the supplementary material Section
\ref{apx:simulation}.  Our simulation results establish that there is
significant potential for improvement over existing methods, and that
the graph-topological aspects of time series analysis are relevant for
both theory and practice.  Concluding remarks on further open problems
and extensions are provided in Section \ref{sec:conclusion}.  The
proofs of each proposition and theorem are also relegated to the
supplementary material, simple corollaries have proofs included in the
main text.

\section{Graph Topological Aspects of Granger causality}
\label{sec:theory}
\subsection{Formal Setting}
Consider the space $L_2(\Omega)$, the usual Hilbert space of finite
variance random variables over a probability space
$(\Omega, \mathcal{F}, \mathbb{P})$ having inner product
$\inner{x}{y} = \E[xy]$.  We will work with a discrete time and
wide-sense stationary (WSS) $n$-dimensional vector valued process
$x(t)$ (with $t \in \Z$) where the $n$ elements take values in $L_2$.
We suppose that $x(t)$ has zero mean, $\E x(t) = 0$, and has
absolutely summable matrix valued covariance sequence
$R(\tau) \overset{\Delta}{=} \E x(t)x(t - \tau)^\T$, and an absolutely continuous
spectral density.

We will also work frequently with the spaces spanned by the values of
such a process

\begin{equation}
  \label{eq:hilbert_space_defn}
  \begin{aligned}
    \H_t^{(x)} &= \cl \{\sum_{\tau = 0}^p a_\tau^\T x(t - \tau)\ |\ a_\tau \in \R^n, p \in \N\} \subseteq L_2(\Omega)\\
    H_t^{(x)} &= \{a x(t)\ |\ a \in \R\} \subseteq L_2(\Omega),
  \end{aligned}
\end{equation}

where the closure is naturally in mean-square.  We will often omit the
superscript $x$ which should be clear from context.  Evidently these
spaces are separable, and as closed subspaces of a Hilbert space they
are themselves Hilbert.  We will denote the spaces generated in
analogous ways by particular components of $x$ as e.g.
$\H_t^{(i, j)}$, $\H_t^{i}$ or by all but a particular component as
$\H_t^{(-j)}$.

As a consequence of the Wold decomposition theorem (see
e.g. \cite{lindquist}), every WSS sequence has the moving average
$\MA(\infty)$ representation

\begin{equation}
  \label{eqn:wold}
  x(t) = c(t) + \sum_{\tau = 0}^\infty A(\tau) v(t - \tau),
\end{equation}

where $c(t)$ is a purely deterministic sequence, $v(t)$ is an
uncorrelated sequence and $A(0) = I$.  We will assume that $c(t) = 0$.
Given our setup, this representation can be inverted to yield the
$\VAR(\infty)$ form

\begin{equation}
  \label{eqn:ar_representation}
  x(t) = \sum_{\tau = 1}^\infty B(\tau) x(t - \tau) + v(t).
\end{equation}

The Equations \eqref{eqn:wold}, \eqref{eqn:ar_representation} can be
represented as $x(t) = \A(z)v(t) = \B(z)x(t) + v(t)$ via the action
(convolution) of the operators (LTI filters)
$$\A(z) \defeq \sum_{\tau = 0}^\infty A(\tau)z^{-\tau}$$ and
$$\B(z) \defeq \sum_{\tau = 1}^\infty B(\tau)z^{-\tau},$$ where the
operator $z^{-1}$ is the back shift operator acting on
$\ell_2^n(\Omega, \mathcal{F}, \mathbb{P})$, that is:

\begin{equation}
  \label{eqn:filter_action}
  \B_{ij}(z)x_j(t) \defeq \sum_{\tau = 1}^\infty B_{ij}(\tau)x_j(t - \tau).
\end{equation}

Finally, we have the inversion formula

\begin{equation}
  \label{eqn:lsi_inversion}
  \A(z) = (I - \B(z))^{-1} = \sum_{k = 0}^\infty \B(z)^k.
\end{equation}

The aforementioned assumptions are quite weak.  The strongest
assumption we require is finally that $\Sigma_v$ is a diagonal
positive-definite matrix, which is referred to as a lack of
instantaneous feedback in $x(t)$.  We formally state our setup as a
definition, which is the setup for the remainder of the paper:

\begin{definition}[Basic Setup]
  \label{def:basic_setup}
  The process $x(t)$ is an $n$ dimensional wide sense stationary
  process having invertible $\VAR(\infty)$ representation
  \eqref{eqn:ar_representation} where $v(t)$ is sequentially
  uncorrelated and has a diagonal positive-definite covariance matrix.
  The $\MA(\infty)$ representation of Equation \eqref{eqn:wold} has
  $c(t) = 0$ and $A(0) = I$.
\end{definition}

\subsection{Granger Causality}

\begin{definition}[Granger Causality]
  \label{def:granger_causality}
  For the WSS series $x(t)$ satisfying the assumptions of Definition
  \ref{def:basic_setup} we will say that component $x_j$
  \textit{Granger-Causes} (GC) component $x_i$ (with respect to $x$)
  and write $x_j \gc x_i$ if

\begin{equation}
  \linEerr{x_i(t)}{\H_{t - 1}} < \linEerr{x_i(t)}{\H^{(-j)}_{t - 1}},
\end{equation}

where $\xi[x \ |\ \H] \defeq \E (x - \linE{x}{\H})^2$ is the mean
squared estimation error and $\linE{x}{\H} \defeq \text{proj}_{\H}(x)$
denotes the (unique) projection onto the Hilbert space $\H$.
\end{definition}

This notion captures the idea that the process $x_j$ provides
information about $x_i$ that is not available from elsewhere.  The
caveat ``with respect to $x$'' is important in that GC relations can
change when components are added to or removed from our collection $x$
of observations, e.g. new GC relations can arise if we remove the
observations of a common cause, and existing GC relations can
disappear if we observe a new mediating series. The notion is closely
related to the information theoretic measure of transfer entropy,
indeed, if the distribution of $v(t)$ is known to be Gaussian then
they are equivalent \cite{barnett2009granger}.

The notion of conditional orthogonality is the essence of
Granger causality, and enables us to obtain results for a fairly
general class of WSS processes, rather than simply $\VAR(p)$ models.

\begin{definition}[Conditional Orthogonality]
  \label{lem:conditional_orthogonality_equivalence}
  Consider three closed subspaces of a Hilbert space $\mathcal{A}$,
  $\mathcal{B}$, $\mathcal{X}$.  We say that $\mathcal{A}$ is
  conditionally orthogonal to $\mathcal{B}$ given $\mathcal{X}$
  and write $\mathcal{A} \perp \mathcal{B}\ |\ \mathcal{X}$ if

    \begin{equation*}
      \inner{a - \linE{a}{\mathcal{X}}}{b - \linE{b}{\mathcal{X}}} = 0\ \forall a \in \mathcal{A}, b \in \mathcal{B}.
    \end{equation*}

  An equivalent condition is that (see \cite{lindquist} Proposition 2.4.2)
  \begin{equation*}
    \linE{\beta}{\mathcal{A} \vee \mathcal{X}} = \linE{\beta}{\mathcal{X}}\ \forall \beta \in \mathcal{B}
  \end{equation*}
\end{definition}

\begin{theorem}[Granger Causality Equivalences]
  \label{thm:granger_causality_equivalences}
  The following are equivalent:

  \begin{enumerate}
    \item{$x_j \ngc x_i$}
    \item{$\forall \tau \in \N_+\ B_{ij}(\tau) = 0$ i.e. $\B_{ij}(z) = 0$}
    \item{$H_t^{i} \perp \H_{t - 1}^{(j)}\ |\ \H_{t - 1}^{(-j)}$}
    \item{$\linE{x_i(t)}{\H_{t - 1}^{(-j)}} = \linE{x_i(t)}{\H_{t - 1}}$}
  \end{enumerate}
\end{theorem}

\subsection{Granger Causality Graphs}
We establish some graph theoretic notation and terminology, collected
formally in definitions for the reader's convenient reference.

\begin{definition}[Graph Theory Review]
  A \textit{graph} $\gcg = (V, \gcge)$ is simply a
  tuple of sets respectively called \textit{nodes} and \textit{edges}.
  Throughout this paper, we have in all cases
  $V = [n] \defeq \{1, 2, \ldots, n\}$.  We will also focus solely on
  \textit{directed} graphs, where the edges
  $\gcge \subseteq V \times V$ are \textit{ordered} pairs.

  A (directed) \textit{path} (of length $r$) from node $i$ to node
  $j$, denoted $\gcgpath{i}{j}$, is a sequence
  $a_0, a_1, \ldots, a_{r - 1}, a_r$ with $a_0 = i$ and $a_r = j$ such
  that $\forall\ 0 \le k \le r\ (a_k, a_{k + 1}) \in \gcge$, and where
  $(a_k, a_{k - 1})$ are \textit{distinct} for $0 \le k < r$.

  A \textit{cycle} is a path of length $2$ or more between a node and
  itself.  An edge between a node and itself $(i, i) \in \gcge$ (which
  we do not consider to be a cycle) is referred to as a \textit{loop}.

  A graph $\gcg$ is a \textit{directed acyclic graph} (DAG) if it is a
  directed graph and does not contain any cycles.
\end{definition}

\begin{definition}[Parents, Grandparents, Ancestors]
  A node $j$ is a \textit{parent} of node $i$ if $(j, i) \in \gcge$.
  The set of all $i$'s parents will be denoted $\pa{i}$, and we
  explicitly exclude loops as a special case, that is,
  $i \not\in \pa{i}$ even if $(i, i) \in \gcge$.

  The set of level $\ell$ \textit{grandparents} of node $i$, denoted
  $\gpn{\ell}{i}$, is the set such that $j \in \gpn{\ell}{i}$ if and
  only if there is a \textit{directed path} of length $\ell$ in $\gcg$
  from $j$ to $i$.  Clearly, $\pa{i} = \gpn{1}{i}$.

  Finally, the set of \textit{level $\ell$ ancestors} of $i$:
  $\ancn{\ell}{i} = \bigcup_{\lambda \le \ell}\gpn{\lambda}{i}$ is the
  set such that $j \in \ancn{\ell}{i}$ if and only if there is a
  directed path of length $\ell$ \textit{or less} in $\gcg$ from $j$
  to $i$.  The set of \textit{all ancestors} of $i$
  (i.e. $\ancn{n}{i}$) is denoted simply $\anc{i}$.

  Recall that we do not allow a node to be it's own parent, although
  unless $\gcg$ is a DAG, a node \textit{can} be it's own ancestor.
  We will occasionally need to explicitly exclude $i$ from $\anc{i}$,
  in which case we will write $\anc{i}\setminus \{i\}$.
\end{definition}

Our principal object of study will be a graph determined by
Granger causality relations as follows.

\begin{definition}[Causality graph]
  We define the Granger causality graph $\gcg = ([n], \gcge)$ to be the directed
  graph formed on $n$ vertices where an edge $(j, i) \in \gcge$ if and
  only if $x_j$ Granger-causes $x_i$ (with respect to $x$).  That is,
  $$(j, i) \in \gcge \iff j \in \pa{i} \iff x_j \gc x_i.$$
\end{definition}

The edges of the Granger causality graph $\gcg$ can be given a general
notion of ``weight'' by associating an edge $(j, i)$ with the
\textit{strictly causal} LTI filter $\B_{ij}(z)$ (see Equation
\eqref{eqn:filter_action}).  Thence, the matrix $\B(z)$ is analogous
to a \textit{weighted adjacency matrix}\footnote{\footnotesize We are using the
  convention that $\B_{ij}(z)$ is a filter with input $x_j$ and output
  $x_i$ so as to write the action of the system as $\B(z)x(t)$ with
  $x(t)$ as a column vector.  This competes with the usual convention
  for adjacency matrices where $A_{ij} = 1$ if there is an edge
  $(i, j)$.  In our case, the sparsity pattern of $\B_{ij}$ is the
  \textit{transposed} conventional adjacency matrix.} for the graph $\gcg$.  And,
in the same way that the $k^{\text{th}}$ power of an adjacency matrix
counts the number of paths of length $k$ between nodes,
$(\B(z)^k)_{ij}$ is a filter isolating the ``action'' of $j$ on $i$ at
a time lag of $k$ steps, this is exemplified in the inversion formula
\eqref{eqn:lsi_inversion}.

From the $\VAR$ representation of $x(t)$ there is clearly a tight
relationship between each node and it's parent nodes, the relationship
is quantified through the sparsity pattern of $B(z)$.  Similarly, the
following proposition is analogous to the definition of feedback free
processes of \cite{caines1975feedback} and provides an interpretation
of the sparsity pattern of $A(z)$ (from the MA representation of
$x(t)$) in terms of the causality graph $\gcg$.

\begin{proposition}[Ancestor Expansion]
  \label{prop:parent_expanding}
  The component $x_i(t)$ of $x(t)$ can be represented in terms of it's
  parents in $\gcg$:

  \begin{equation}
    \label{eqn:parent_expansion}
    x_i(t) = v_i(t) + \B_{ii}(z)x_i(t) + \sum_{k \in \pa{i}}\B_{ik}(z)x_k(t).
  \end{equation}

  Moreover, $x_i$ can be expanded in terms of it's ancestor's $v(t)$
  components only:

  \begin{equation}
    \label{eqn:ancestor_expansion}
    x_i(t) = \A_{ii}(z)v_i(t) + \sum_{\substack{k \in \anc{i} \\ k \ne i}}\A_{ik}(z)v_k(t),
  \end{equation}

  where $\A(z) = \sum_{\tau = 0}^\infty A(\tau)z^{-\tau}$ is the filter from
  the Wold decomposition representation of $x(t)$, Equation
  \eqref{eqn:wold}.
\end{proposition}

This statement is ultimately about the sparsity pattern in the Wold
decomposition matrices $A(\tau)$ since
$x_i(t) = \sum_{\tau = 0}^\infty \sum_{j = 1}^n A_{ij}(\tau)v_j(t -
\tau)$.  The proposition states that if $j \not \in \anc{i}$ then
$\A_{ij}(z) = 0$.  

\subsection{Pairwise Granger Causality}
\label{sec:pwgc}
Recall that Granger causality in general must be understood with
respect to a particular universe of observations.  If $x_j \gc x_i$
with respect to $x_{-k}$, it may not hold with respect to $x$.  For
example, $x_k$ may be a common ancestor which when observed, completely
explains the connection from $x_j$ to $x_i$.  In this section we study
\textit{pairwise} Granger causality, and seek to understand when
knowledge of pairwise relations is sufficient to deduce the true fully
conditional relations of $\gcg$.

\begin{definition}[Pairwise Granger causality]
  We will say that $x_j$ pairwise Granger-causes $x_i$ and write
  $x_j \pwgc x_i$ if $x_j$ Granger-causes $x_i$ with respect only to
  $(x_i, x_j)$.
\end{definition}

This notion is of interest for a variety of reasons.  From a purely
conceptual standpoint, we will see how the notion can in some sense
capture the idea of ``flow of information'' in the underlying graph,
in the sense that if $j \in \anc{i}$ we expect that $j \pwgc i$.  It may
also be useful for reasoning about the conditions under which
\textit{unobserved} components of $x(t)$ may or may not interfere with
inference in the actually observed components.  Finally, motivated
from a practical standpoint to analyze causation in large systems,
practical estimation procedures based purely on pairwise causality
tests are of interest since the computation of such pairwise relations
is substantially easier.

The following propositions are essentially lemmas used for the proof
of the upcoming Proposition \ref{prop:ancestor_properties}, but
remain relevant for providing intuitive insight into the problems at
hand.

\begin{proposition}
  Consider distinct nodes $i, j$ in a Granger causality graph
  $\gcg$.  If

  \begin{enumerate}[label=(\alph*)]
    \item{$j \not\in \anc{i}$ and $i \not\in \anc{j}$}
    \item{$\anc{i}\cap\anc{j} = \emptyset$}
  \end{enumerate}

  then $\H_t^{(i)} \perp \H_t^{(j)}$, that is,
  $\forall s, \tau \in \Z_+\ \E[x_i(t - s)x_j(t - \tau)] = 0$.  Moreover,
  this means that $j \npwgc i$ and $\linE{x_j(t)}{\H_t^{(i)}} = 0$.
\end{proposition}

\begin{remark}
  It is possible for components of $x(t)$ to be correlated at some
  time lags without resulting in pairwise causality.  For instance,
  the conclusion $j \npwgc i$ of Proposition
  \ref{prop:separated_ancestor_uncorrelated} will still hold even if
  $i \in \anc{j}$, since $j$ cannot provide any information about $i$
  that is not available from observing $i$ itself.
\end{remark}

\begin{proposition}
  Consider distinct nodes $i, j$ in a Granger causality graph $\gcg$.
  If

  \begin{enumerate}[label=(\alph*)]
    \item{$j \not\in \anc{i}$}
    \item{$\anc{i}\cap\anc{j} = \emptyset$}
  \end{enumerate}

  then $j \npwgc i$.
\end{proposition}

The previous result can still be strengthened significantly; notice
that it is possible to have some $k \in \anc{i} \cap \anc{j}$ where still
$j \npwgc i$, an example is furnished by the three node graph
$k \rightarrow i \rightarrow j$ where clearly
$k \in \anc{i}\cap\anc{j}$ but $j \npwgc i$.  We must introduce the concept
of a \textit{confounding} variable, which effectively eliminates the
possibility presented in this example.

\begin{definition}[Confounder]
  A node $k$ will be referred to as a \textit{confounder} of nodes
  $i, j$ (neither of which are equal to $k$) if
  $k \in \anc{i} \cap \anc{j}$ and there exists a path
  $\gcgpath{k}{i}$ not containing $j$, and a path $\gcgpath{k}{j}$ not
  containing $i$. A simple example is furnished by the ``fork'' graph
  $i \leftarrow k \rightarrow j$.
\end{definition}

\begin{proposition}
  \label{prop:ancestor_properties}
  If in a Granger causality graph $\gcg$ where $j \pwgc i$ then
  $j \in \anc{i}$ or $\exists k \in \anc{i} \cap\anc{j}$ which is a
  confounder of $(i, j)$.
\end{proposition}

\begin{remark}
  The interpretation of this proposition is that for $j \pwgc i$ then
  there must either be ``causal flow'' from $j$ to $i$
  ($j \in \anc{i}$) or there must be a confounder $k$ through which
  common information is received.
\end{remark}

An interesting corollary is the following:

\begin{corollary}
  If the graph $\gcg$ is a DAG then $j \pwgc i, i \pwgc j \implies \exists k \in \anc{i} \cap \anc{j}$ confounding $(i, j)$.
\end{corollary}

It seems reasonable to expect a converse of Proposition
\ref{prop:ancestor_properties} to hold, i.e.
$j \in \anc{i} \Rightarrow j \pwgc i$.  Unfortunately, this is not the
case in general, as different paths through $\gcg$ can lead to
cancellation (see Figure \ref{fig:diamond_cancellation}).  In fact, we
do not even have $j \in \pa{i} \Rightarrow j \pwgc i$ (see Figure
\ref{fig:lag_cancellation}).

\begin{example}
  \label{ex:diamond_cancellation}
  Firstly, on $n = 4$ nodes, ``diamond'' shapes can lead to cancellation on paths of length 2:

\begin{equation*}
  x(t) =
  \left[
    \begin{array}{cccc}
      0 & 0 & 0 & 0\\
      a & 0 & 0 & 0\\
      -a & 0 & 0 & 0\\
      0 & 1 & 1 & 0\\
    \end{array}
  \right] x(t - 1) + v(t),
\end{equation*}

with $\E v(t) = 0,\ \E v(t)v(t - \tau)^\T = \delta_\tau I$.

By directly calculating

\begin{align*}
  x_4(t) &= x_2(t - 1) + x_3(t - 1) + v_4(t)\\
         &= ax_1(t - 2) + av_2(t - 1) - ax_1(t - 2) -av_3(t - 1) + v_4(t)\\
         &= a(v_2(t - 1) - v_3(t - 1)) + v_4(t),
\end{align*}

we see that, since $v(t)$ is isotropic white noise, $1 \npwgc 4$.  The problem here is that there are multiple paths from $x_1$ to $x_4$.
\end{example}

\begin{figure}
  \centering
  \caption{Examples illustrating the difficulty of obtaining a converse to Proposition \ref{prop:ancestor_properties}}
  \begin{subfigure}[b]{0.35\textwidth}
    \includegraphics[width=\linewidth]{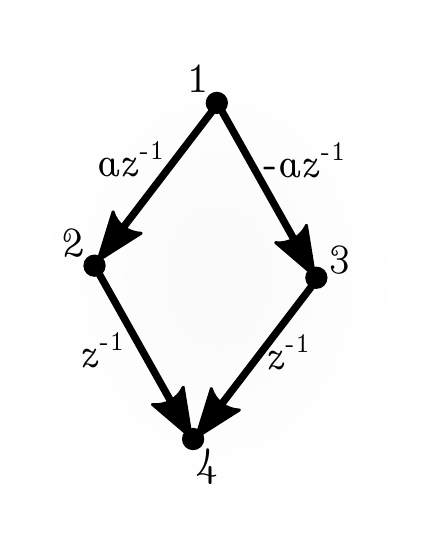}
    \caption{Path cancellation: $j \in \anc{i} \nRightarrow j \pwgc i$}
    \label{fig:diamond_cancellation}
  \end{subfigure}
  \begin{subfigure}[b]{0.35\textwidth}
    \includegraphics[width=\linewidth]{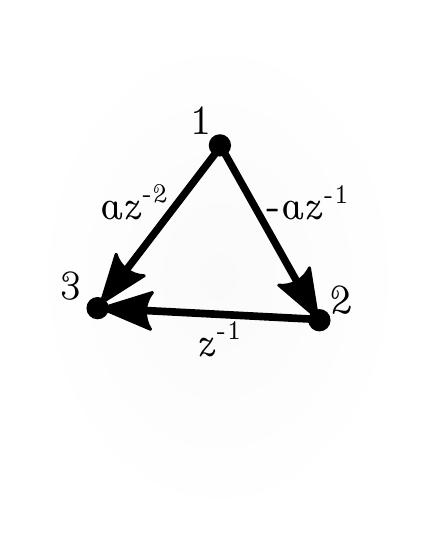}
    \caption{Cancellation from time lag: $j \in \pa{i} \nRightarrow j \pwgc i$}
    \label{fig:lag_cancellation}
  \end{subfigure}
\end{figure}

\begin{example}
  \label{ex:lag_cancellation}
  A second example on $n = 3$ nodes is also worth examining, in this case
  cancellation is a result of differing time lags.

\begin{equation*}
  x(t) =
  \left[
    \begin{array}{ccc}
      0 & 0 & 0\\
      -a & 0 & 0\\
      0 & 1 & 0\\
    \end{array}
  \right] x(t - 1) +
  \left[
    \begin{array}{ccc}
      0 & 0 & 0\\
      0 & 0 & 0\\
      a & 0 & 0\\
    \end{array}
  \right] x(t - 2) + v(t)
\end{equation*}

Then

\begin{align*}
  x_2(t) &= v_2(t) - ax_1(t - 1)\\
  x_3(t) &= v_3(t) + x_2(t - 1) + ax_1(t - 2)\\
  \Rightarrow x_3(t) &= v_2(t - 1) + v_3(t),
\end{align*}

and again $1 \npwgc 3$.
\end{example}

\subsection{Strongly Causal Graphs}
\label{sec:strongly_causal_graphs}

In this section and the next we will seek to understand when converse
statements of Proposition \ref{prop:ancestor_properties} \textit{do}
hold.  One possibility is to restrict the coefficients of the system
matrix, e.g. by requiring that $B_{ij}(\tau) \ge 0$.  Instead,
we think it more meaningful to focus on the defining feature of
time series networks, that is, the topology of $\gcg$.

\begin{definition}[Strongly Causal]
  \label{def:strongly_causal}
  We will say that a Granger causality graph $\gcg$ is
  \textit{strongly causal} if there is at most one directed path between
  any two nodes.  Strongly Causal Graphs will be referred to as SCGs.
\end{definition}

Examples of strongly causal graphs include directed trees (or
forests), DAGs where each node has at most one parent, and Figure
\ref{fig:example_fig3} of this paper.  A complete bipartite graph with
$2n$ nodes is also strongly causal, demonstrating that the number of
edges of such a graph can still scale quadratically with the number of
nodes.  It is evident that the strong causal property is inherited by
subgraphs.

\begin{example}
  Though examples of SCGs are easy to construct in theory, should
  practitioners expect SCGs to arise in application?  While a positive
  answer to this question is not \textit{necessary} for the concept to
  be useful, it is certainly sufficient.  Though the answer is likely
  to depend upon the particular application area, examples appear to
  be available in biology, in particular, the authors of
  \cite{discovering_graphical_Granger_causality_using_the_truncating_lasso_penalty}
  cite an example of the so called ``transcription regulatory network
  of \textit{E.coli}'', and
  \cite{learning_genome_scale_regulatory_networks} study a much larger
  regulatory network of \textit{Saccharomyces cerevisiae}.  These
  networks, which we reproduce in Figure \ref{fig:gene_networks},
  appear to have at most a small number of edges which violate the
  strong-causality condition.

  \begin{figure}[h]
    \centering
    \caption{Transcription Regulatory Networks}
    \label{fig:gene_networks}
    \begin{subfigure}[b]{0.45\textwidth}
      \caption{\textit{E.Coli} Network
        \cite{discovering_graphical_Granger_causality_using_the_truncating_lasso_penalty}}
      \label{fig:gene_network1}
      \includegraphics[width=\linewidth, height=\linewidth]{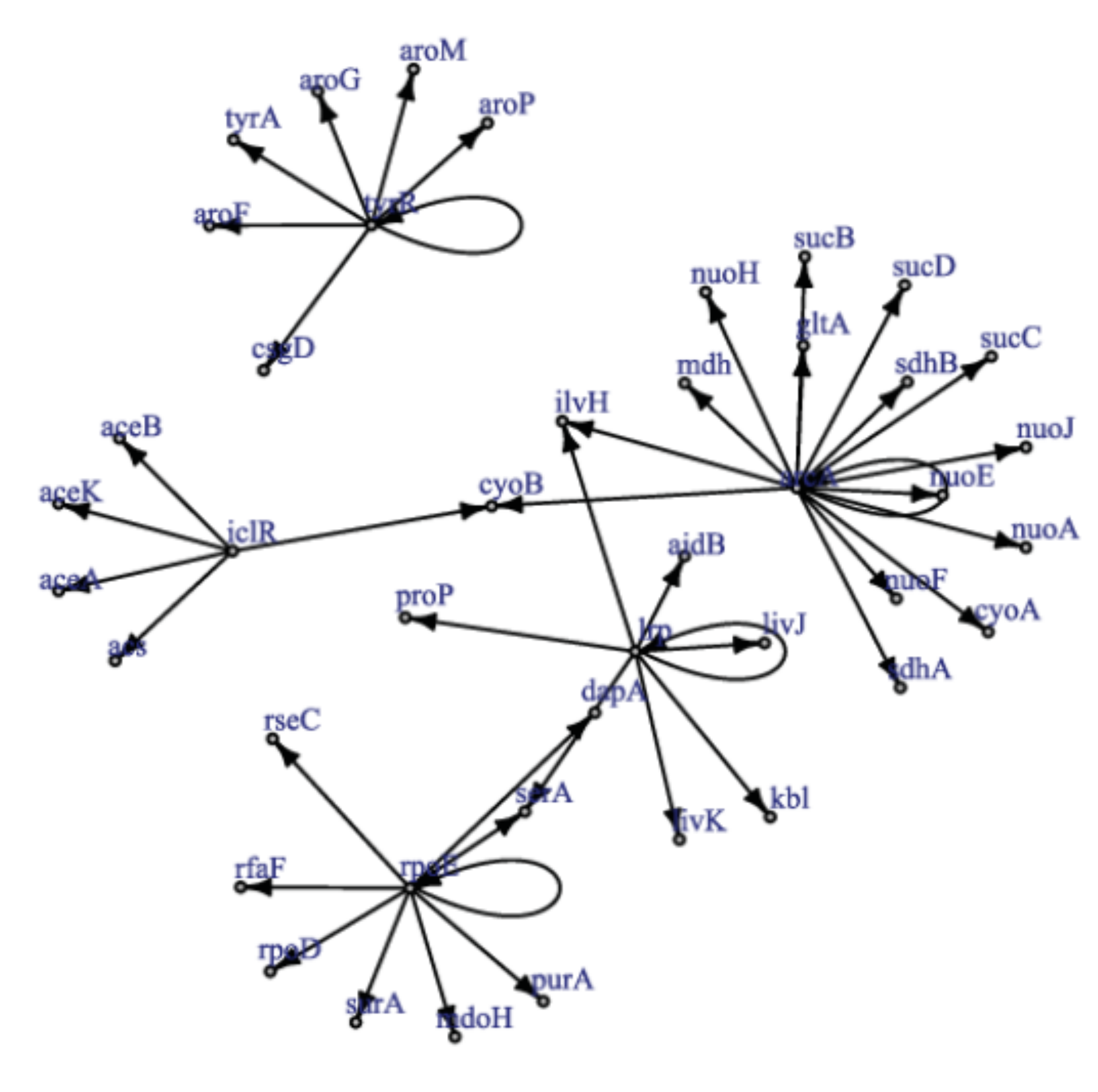}
    \end{subfigure}
    \begin{subfigure}[b]{0.45\textwidth}
      \caption{\textit{Saccharomyces cerevisiae} Network
        \cite{learning_genome_scale_regulatory_networks}}
      \label{fig:gene_network2}
      \includegraphics[width=\linewidth, height=\linewidth]{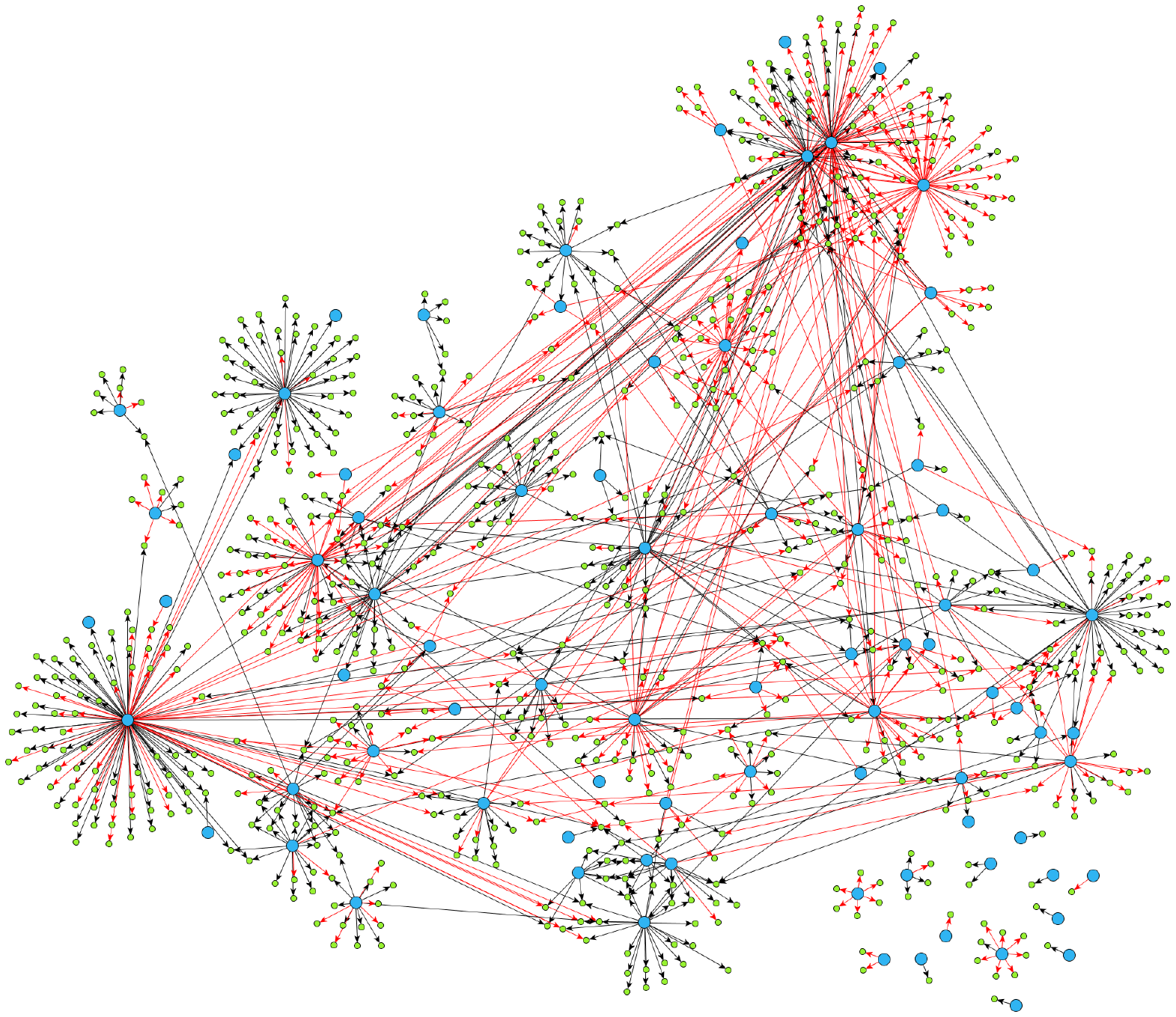}
    \end{subfigure}

    {\footnotesize Figure \ref{fig:gene_network1} has only one edge
      violating the strong-causality assumption, and Figure
      \ref{fig:gene_network2} appears qualitatively to be nearly
      strongly causal.  In particular, most of the edges are emanating
      from network hubs, and many of the other edges are colliders.

      Figure \ref{fig:gene_network1}
      is reproduced under the Creative Commons Attribution
      Non-Commercial License
      (http://creativecommons.org/licenses/by-nc/2.5) and Figure
      \ref{fig:gene_network2} under the Creative Commons Attribution
      License (https://creativecommons.org/licenses/by/4.0/)}
  \end{figure}
\end{example}

For later use, and to get a feel for the topological implications of
strong causality, we explore a number of properties of such graphs
before moving into the main result of this section.  The following
important property essentially strengthens Proposition
\ref{prop:ancestor_properties} for the case of strongly causal graphs.

\begin{proposition}
  \label{prop:sc_graph_common_anc}
  In a strongly causal graph if $j \in \anc{i}$ then any
  $k \in \anc{i} \cap \anc{j}$ is not a confounder, that is,
  the unique path from $k$ to $i$ contains $j$.
\end{proposition}

\begin{corollary}
  \label{cor:parent_corollary}
  If $\gcg$ is a strongly causal DAG then $i \pwgc j$ and $j \in \anc{i}$ are
  \textit{alternatives}, that is $i \pwgc j \Rightarrow j \notin \anc{i}$.
\end{corollary}
\begin{proof}
  Suppose that $i \pwgc j$ and $j \in \anc{i}$.  Then since $\gcg$ is
  acyclic $i \not\in \anc{j}$, and by Proposition
  \ref{prop:ancestor_properties} there is some
  $k \in \anc{i}\cap\anc{j}$ which is a confounder.  However, by
  Proposition \ref{prop:sc_graph_common_anc} $k$ cannot be a
  confounder, a contradiction.
\end{proof}



  
\begin{corollary}
  \label{cor:bidirectional_edge}
  If $\gcg$ is a strongly causal DAG such that $i \pwgc j$ and
  $j \pwgc i$, then $i \not\in \anc{j}$ and $j \not\in \anc{i}$.  In
  particular, a pairwise bidirectional edge indicates the absence of
  any edge in $\gcg$.
\end{corollary}
\begin{proof}
  This follows directly from applying Corollary
  \ref{cor:parent_corollary} to $i \pwgc j$ and $j \pwgc i$.
\end{proof}

In light of Proposition \ref{prop:sc_graph_common_anc}, the following
provides a partial converse to Proposition
\ref{prop:ancestor_properties}, and supports the intuition of ``causal
flow'' through paths in $\gcg$.

\begin{proposition}
  \label{prop:pwgc_anc}
  If $\gcg$ is a strongly causal DAG then $j \in \anc{i} \Rightarrow j \pwgc i$.
\end{proposition}

We immediately obtain the corollary, which we remind the reader is,
surprisingly, not true in a general graph.

\begin{corollary}
  \label{cor:gc_implies_pwgc}
  If $\gcg$ is a strongly causal DAG then $j \gc i \Rightarrow j \pwgc i$.
\end{corollary}

\begin{example}
  As a final remark of this subsection we note that a complete
  converse to Proposition \ref{prop:ancestor_properties} is not
  possible without additional conditions.  Consider the ``fork''
  system on $3$ nodes (i.e. $2 \leftarrow 1 \rightarrow 3$) defined by

  \begin{equation*}
    x(t) =
    \left[
      \begin{array}{cccc}
        0 & 0 & 0\\
        a & 0 & 0\\
        a & 0 & 0\\
      \end{array}
    \right] x(t - 1) + v(t).
  \end{equation*}

  In this case, node $1$ is a confounder for nodes $2$ and $3$, but
  $x_3(t) = v_3(t) - v_2(t) + x_2(t)$ and $2 \npwgc 3$ (even
  though $x_2(t)$ and $x_3(t)$ are contemporaneously correlated)

  If we were to augment this system by simply adding an autoregressive
  component (i.e. some ``memory'') to $x_1(t)$ e.g.
  $x_1(t) = v_1(t) + b x_1(t - 1)$ then we \textit{would} have
  $2 \pwgc 3$ since then
  $x_3(t) = v_3(t) + av_1(t - 1) - bv_2(t - 1) + bx_2(t - 1)$.  We
  develop this idea further in the next section.
\end{example}

\subsection{Persistent Systems}
\label{sec:persistent_systems}
In section \ref{sec:strongly_causal_graphs} we obtained a converse to
part $(a)$ of Proposition \ref{prop:ancestor_properties} via the
notion of a strongly causal graph topology (see Proposition
\ref{prop:pwgc_anc}).  In this section, we study conditions under
which a converse to part $(b)$ will hold.

\begin{definition}[Lag Function]
  Given a causal filter $\B(z) = \sum_{\tau = 0}^\infty b(\tau)z^{-\tau}$
  define 

  \begin{align}
    \tau_0(\B) &= \text{inf }\{\tau \in \Z_+\ |\ b(\tau) \ne 0\},\\
    \tau_{\infty}(\B) &= \text{sup }\{\tau \in \Z_+\ |\ b(\tau) \ne 0\}.
  \end{align}

  i.e. the ``first'' and ``last'' coefficients of the filter $\B(z)$,
  where $\tau_\infty(\B) \defeq \infty$ if the filter has an infinite
  length, and $\tau_0(\B) \defeq \infty$ if $\B(z) = 0$.
\end{definition}

\begin{definition}[Persistent]
  We will say that the process $x(t)$ with Granger causality graph
  $\gcg$ is \textit{persistent} if for every $i \in [n]$ and every
  $k \in \anc{i}$ we have $\tau_0(\A_{ik}) < \infty$ and $\tau_\infty(\A_{ik}) = \infty$.
\end{definition}

\begin{remark}
  In the context of Granger causality, ``most'' systems should be
  persistent.  In particular, $\mathsf{\VAR}(p)$ models are likely to
  be persistent since these naturally result in an equivalent
  $\mathsf{MA}(\infty)$ representation, see Example
  \ref{ex:persistent_system}.

  Moreover, persistence is not the weakest condition necessary for the
  results of this section, the condition that for each $i, j$ there is
  some $k \in \anc{i}\cap\anc{j}$ such that
  $\tau_0(\A_{jk}) < \tau_\infty(\A_{ik})$ is enough.  The intuition
  being that nodes $i$ and $j$ are not receiving temporally disjoint
  information from $k$.

  The etymology for the persistence condition can be explained by
  supposing that the two nodes $i, j$ each have a loop (i.e.
  $\B_{ii}(z) \ne 0, \B_{jj}(z) \ne 0$) then this autoregressive
  component acts as ``memory'', and so the influence from the
  confounder $k$ \textit{persists} in $x_i(t)$, and
  $\tau_\infty(\A_{ik}) = \infty$ for each confounder is expected.
\end{remark}

\begin{example}
  \label{ex:persistent_system}
  Consider a process $x(t)$ generated by the $\VAR(1)$
  model\footnote{Recall that any $\VAR(p)$ model with $p < \infty$ can
    be written as a $\VAR(1)$ model, so we lose little generality in
    considering this case.}  having $\B(z) = Bz^{-1}$.  If $B$ is
  diagonalizable, and has at least $2$ distinct eigenvalues, then
  $x(t)$ is persistent.

  See the supplementary material for an analysis of this example.
\end{example}

In order to eliminate the possibility of a particular sort of
cancellation, an ad-hoc assumption is required.  Strictly speaking,
the persistence condition is not a necessary or sufficient condition
for the following, but cases where the following fails to hold, and
persistence \textit{does} hold, are unavoidable pathologies.

\begin{assumption}
  \label{ass:T_causality}
  Fix $i, j \in [n]$ and let $H_i(z)$ be the strictly-causal filter
  such that

  \begin{equation*}
    H_i(z)x_i(t) = \linE{x_i(t)}{\H_{t - 1}^{(i)}},
  \end{equation*}

  and similarly for $H_j(z)$.  Then define

  \begin{equation}
    \label{eqn:T_filter}
    T_{ij}(z) = \sum_{k \in \anc{i} \cap \anc{j}}\sigma_k^2\A_{ik}(z^{-1})(1 - H_i(z^{-1}))(1 - H_j(z))\A_{jk}(z),
  \end{equation}

  where $\sigma_k^2 = \E v_k(t)^2$.

  We will say that Assumption \ref{ass:T_causality} is satisfied if
  for every $i, j \in [n]$, $T_{ij}(z)$ is either constant over $z$
  (i.e. each $z^k$ coefficient for $k \in \Z \setminus \{0\}$ is $0$),
  or is \textit{neither} causal (i.e. containing only $z^{-k}$ terms,
  for $k \ge 0$) \textit{or} anti-causal (i.e. containing only $z^k$
  terms, for $k \ge 0$).  Put succinctly, $T_{ij}(z)$ must be
  two-sided.
\end{assumption}

\begin{remark}
  Under the condition of persistence, the only way for Assumption
  \ref{ass:T_causality} to fail is through cancellation in the terms
  defining $T_{ij}(z)$.  For example, the condition is assured if
  $x(t)$ is persistent, and there is only a single confounder.
  Unfortunately, some pathological behaviour resulting from
  confounding nodes seems to be unavoidable without some assumptions
  about the parameters of the $\MA(\infty)$ system defining $x(t)$.
\end{remark}

\begin{proposition}
  \label{prop:persistence_converse}
  Fix $i, j \in [n]$ and suppose $\exists k \in \anc{i} \cap \anc{j}$
  which confounds $i, j$.  Then, if $T_{ij}(z)$ is not causal we have
  $j \pwgc i$, and if $T_{ij}(z)$ is not anti-causal we have
  $i \pwgc j$.  Moreover, if Assumption \ref{ass:T_causality} is
  satisfied, then $j \pwgc i \iff i \pwgc j$.
\end{proposition}

\begin{remark}
  The importance of this result is that when $i \pwgc j$ is a result
  of a confounder $k$, then $i \pwgc j \iff j \pwgc i$.  This
  implies that in a strongly causal graph every bidirectional pairwise
  causality relation must be the result of a confounder.  Therefore,
  in a strongly causal graph, pairwise causality analysis is
  \textit{immune to confounding} (since we can safely remove all
  bidirectional edges).
\end{remark}

\subsection{Recovering $\gcg$ via Pairwise Tests}
\label{sec:pairwise_algorithm}
We arrive at the main conclusion of the theoretical analysis in this
paper.

\begin{theorem}[Pairwise Recovery]
  \label{thm:scg_recovery}
  If the Granger causality graph $\gcg$ for the process $x(t)$ is a
  strongly causal DAG and Assumption \ref{ass:T_causality} holds, then
  $\gcg$ can be inferred from pairwise causality tests.  The procedure
  can be carried out, assuming we have an oracle for pairwise
  causality, via Algorithm (\ref{alg:pwgr}).
\end{theorem}

\begin{algorithm}
  \SetKwInOut{Input}{input}
  \SetKwInOut{Output}{output}
  \SetKwInOut{Initialize}{initialize}
  \DontPrintSemicolon

  \caption{Pairwise Granger Causality Algorithm}
  \label{alg:pwgr}
  \TitleOfAlgo{Pairwise Graph Recovery}
  \Input{Pairwise Granger causality relations between a persistent
  process of dimension $n$ whose joint Granger causality
  relations are known to form a strongly causal DAG $\gcg$.}

  \Output{Edges $\gcge = \{(i, j) \in [n] \times [n]\ |\ i \gc j \}$ of
    the graph $\gcg$.}
  \Initialize{$S_0 = [n]$  \texttt{\# unprocessed nodes}\\
    $E_0 = \emptyset$  \texttt{\# edges of }$\gcg$\\

    $k = 1$ \texttt{\# a counter used only for notation}}
  \BlankLine
  $W \leftarrow \{(i, j)\ |\ i \pwgc j, j \npwgc i\}$  \texttt{\# candidate edges}\\
  $P_0 \leftarrow \{i \in S_0\ |\ \forall s \in S_0\ (s, i) \not\in W\}$  \texttt{\# parent-less nodes}\\
  \While{$S_{k - 1} \ne \emptyset$}{
    $S_k \leftarrow S_{k - 1} \setminus P_{k - 1}$ \texttt{\# remove nodes with depth }$k - 1$\\
    $P_k \leftarrow \{i \in S_k\ |\ \forall s \in S_k\ (s, i) \not\in W\}$   \texttt{\# candidate children}\\

    $D_{k0} \leftarrow \emptyset$\\
    \For{$r = 1, \ldots, k$} 
    {
      $Q \leftarrow E_{k - 1} \cup \big(\bigcup_{\ell = 0}^{r - 1} D_{k\ell}\big)$ \texttt{\# currently known edges}\\
      $D_{kr} \leftarrow \{(i, j) \in P_{k - r} \times P_k\ |\ (i, j) \in W,\ \text{no } i \rightarrow j \text{ path in } Q\}$
    }
    $E_k \leftarrow E_{k - 1} \cup \big(\bigcup_{r = 1}^k D_{kr}\big)$ \texttt{\# update } $E_k$ \texttt{ with new edges}\\
    $k \leftarrow k + 1$
  }
  \Return{$E_{k - 1}$}
\end{algorithm}

The theorem is proven in the supplementary material by establishing
the correctness of Algorithm (\ref{alg:pwgr}).  The idea is to
iteratively ``peel away layers'' of nodes by removing the nodes that
have no parents remaining.  The requirement of strong causality
ensures that all actual edges of $\gcg$ manifest in some way as
pairwise relations (by Proposition \ref{prop:pwgc_anc}), and the
no-cancellation condition of Assumption \ref{ass:T_causality} allows
confounding to be eliminated by removing bidirectional edges (by
Proposition \ref{prop:persistence_converse} and Corollary
\ref{cor:bidirectional_edge}).  Without Assumption
\ref{ass:T_causality}, then each confounded pair would give rise to
$4$ possible pairwise topologies consistent with $\gcg$, one for each
type of pairwise edge (no edge, unidirectional, bidirectional).

\begin{example}
  The set $W$ collects ancestor relations in $\gcg$ (see Lemma
  \ref{lem:W_subset_E}).  In reference to Figure
  \ref{fig:example_fig3}, each of the solid black edges, as well as
  the dotted red edges will be included in $W$, but \textit{not} the
  bidirectional green dash-dotted edges, which we are able to exclude
  as results of confounding.  The groupings $P_0, \ldots, P_3$ are also
  indicated in Figure \ref{fig:example_fig3}.

  The algorithm proceeds first with the parent-less nodes $1, 2$ on the
  initial iteration where the edge $(1, 3)$ is added to $E$.  On the
  next iteration, the edges $(3, 4), (2, 4), (3, 5)$ are added, and
  the false edges $(1, 4), (1, 5)$ are excluded due to the paths
  $1 \rightarrow 3 \rightarrow 4$ and $1 \rightarrow 3 \rightarrow 5$
  already being present.  Finally, edge $(4, 6)$ is added, and the false
  $(1, 6), (3, 6), (2, 6)$ edges are similarly excluded due to the
  ordering of the inner loop.
  
  \begin{figure}
    \centering
    \caption{Example graph for Algorithm \ref{alg:pwgr}}
    {\footnotesize{Black arrows indicate true parent-child
        relations.  Red dotted arrows indicate pairwise causality (due to
        non-parent relations), green dash-dotted arrows indicate
        bidirectional pairwise causality (due to the confounding node
        $1$).  Blue groupings indicate each $P_k$ in Algorithm
        \ref{alg:pwgr}.}}
    \label{fig:example_fig3}
    
    \begin{subfigure}[b]{0.45\textwidth}
      \includegraphics[width=\linewidth]{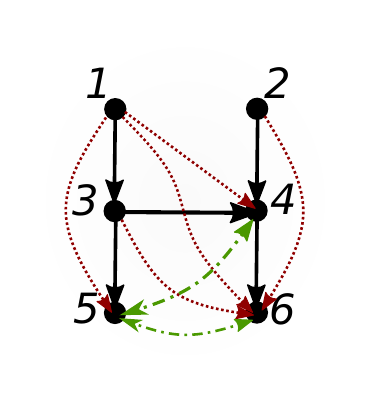}
    \end{subfigure}
    \begin{subfigure}[b]{0.45\textwidth}
      \includegraphics[width=\linewidth]{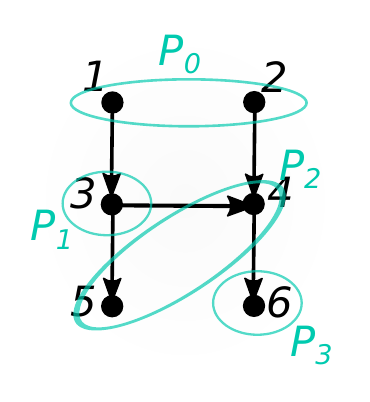}
    \end{subfigure}
  \end{figure}

  That we need to proceed backwards through $P_{k - r}$ as in the
  inner loop on $r$ can also be seen from this example, where if
  instead we simply added the set

  \begin{equation*}
    D_k' = \{(i, j) \in \Big(\bigcup_{r = 1}^k P_{k - r}\Big) \times P_k\ |\ i \pwgc j \}
  \end{equation*}

  to $E_k$ then we would infer the false positive edge
  $1 \rightarrow 4$.  Moreover, the same example shows that simply
  using the set

  \begin{equation*}
    D_k'' = \{(i, j) \in P_{k - 1} \times P_k\ |\ i \pwgc j \}  ,
  \end{equation*}

  causes the edge $1 \rightarrow 3$ to be missed.
\end{example}

\section{Simulation}
\label{sec:simulation}
We implement an heuristic inspired by Algorithm \ref{alg:pwgr} by
replacing the population statistics with finite sample tests, the
details of which can be found in the supplementary material Section
\ref{sec:finite_pwgc} (see Algorithm \ref{alg:finite_pwgc}).  The
heuristic is essentially controlling the false discovery rate
substantially below what it would be with a threshold based pairwise
scheme.  The methods are easily parallelizable, and can scale to
graphs with thousands of nodes on a single machine.  By contrast,
scaling the LASSO to this large of a network (millions of variables)
is nontrivial and extremely computationally demanding.

We run experiments using two separate graph topologies having $n = 50$
nodes: a strongly causal graph (SCG) and a directed acyclic graph
(DAG).  Consult Section \ref{apx:simulation} for details on how data
is generated from these models.

We compare our results against the adaptive LASSO
\cite{adaptive_lasso_zou2006}, which outperformed both the LASSO and
the grouped LASSO by a large margin.  Motivated by scaling, we split the squared error
term into separate terms, one for each group of incident edges on a
node, and estimate the collection of $n$ incident filters
$\big\{\B_{ij}(z)\big\}_{j = 1}^n$ that minimizes
$\xi_i^{\text{LASSO}}$ in the following:

\begin{equation}
  \begin{aligned}
  \xi_i^{\text{LASSO}}(\lambda) &= \underset{B}{\text{min}}\ \frac{1}{T}\sum_{t = p_{\text{max}} + 1}^T\big(x_i(t) - \sum_{\tau = 1}^{p_{\text{max}}}\sum_{j = 1}^n B_{i, j}(\tau) x(t - \tau)\big)^2 + \lambda \sum_{\tau = 1}^{p_{\text{max}}} \sum_{j = 1}^n |B_{ij}(\tau)|\\
  \xi_i^{\text{LASSO}} &= \underset{\lambda \ge 0}{\text{min}}\ \xi_i^{\text{LASSO}}(\lambda) + \mathsf{BIC}\big(B_i^{\text{LASSO}}(\lambda)\big)\\
  \end{aligned}
\end{equation}

where we are choosing $\lambda$, the regularization parameters, via the BIC.
This is similar to the work of \cite{arnold2007temporal}, except that
we have replacing the LASSO with the Adaptive LASSO, which provides
dramatically superior performance.

\begin{figure}
  \centering
  \caption{Representative Random Graph Topologies on $n = 50$ Nodes}
  \label{fig:random_graph_topologies}
  \begin{subfigure}[b]{0.3\textwidth}
    \caption{SCG}
    \includegraphics[width=\linewidth]{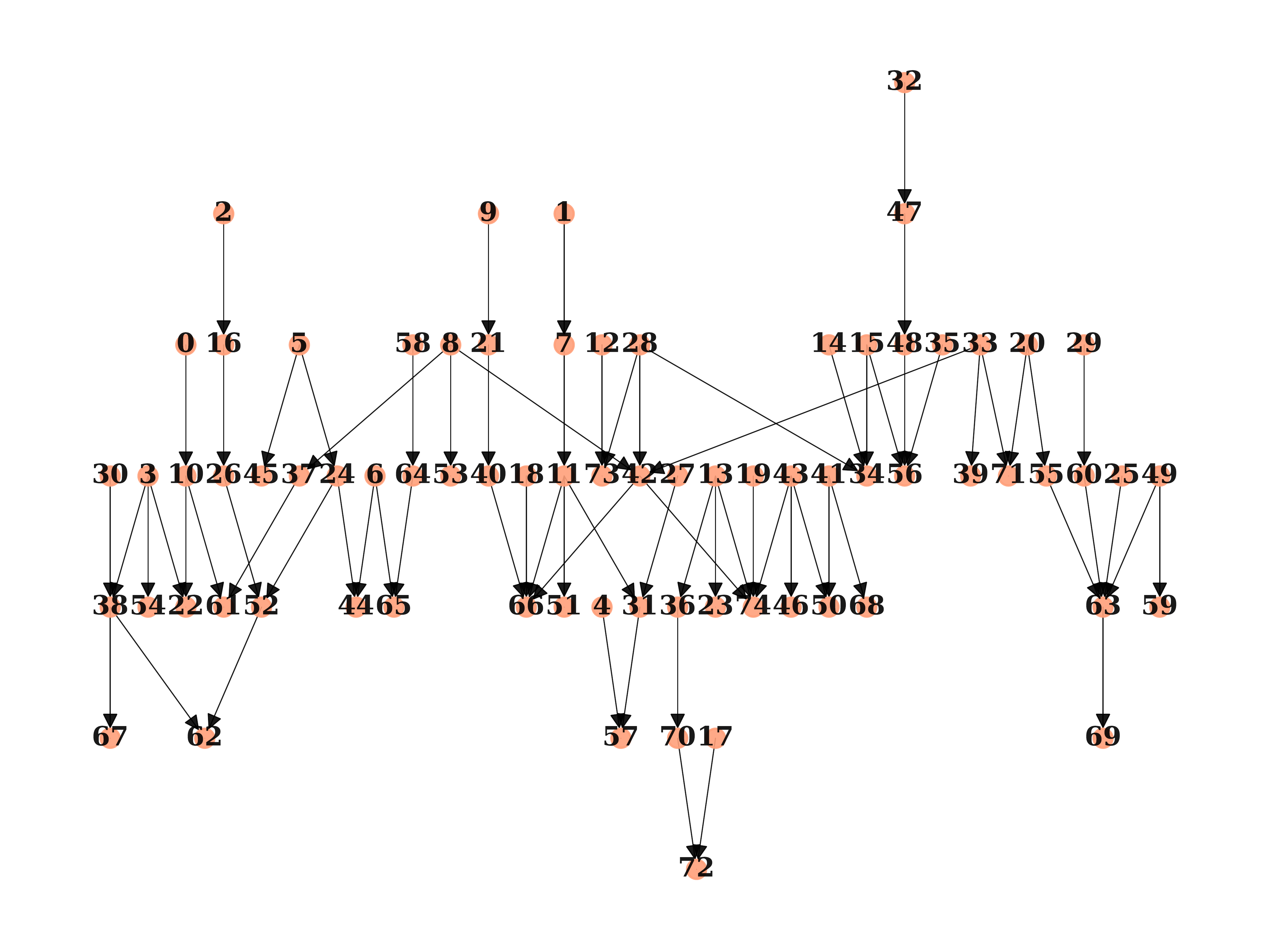}
  \end{subfigure}
  \begin{subfigure}[b]{0.3\textwidth}
    \caption{DAG $(q = \frac{2}{n})$}
    \includegraphics[width=\linewidth]{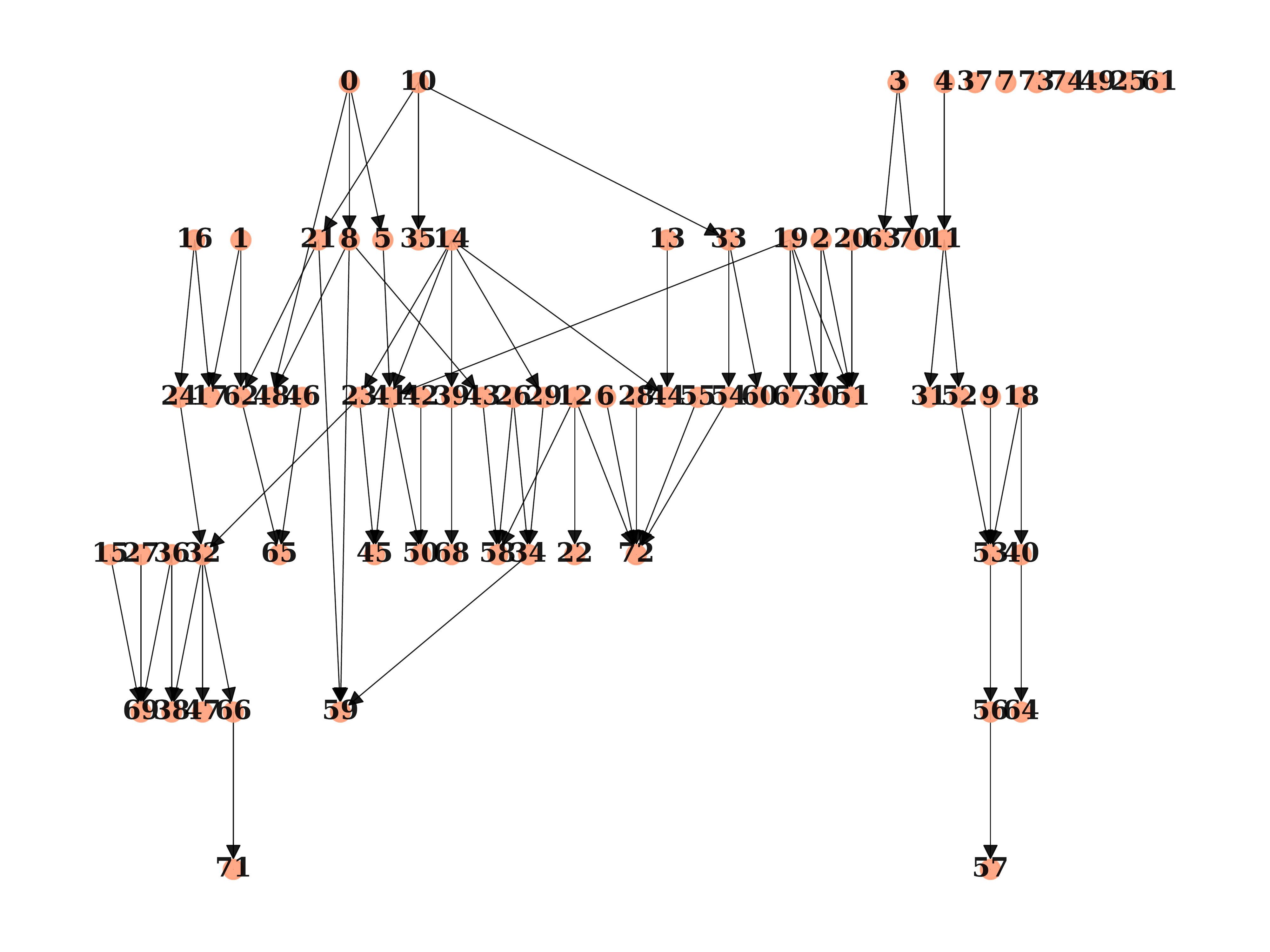}
  \end{subfigure}
  \begin{subfigure}[b]{0.3\textwidth}
    \caption{DAG $(q = \frac{4}{n})$}
    \includegraphics[width=\linewidth]{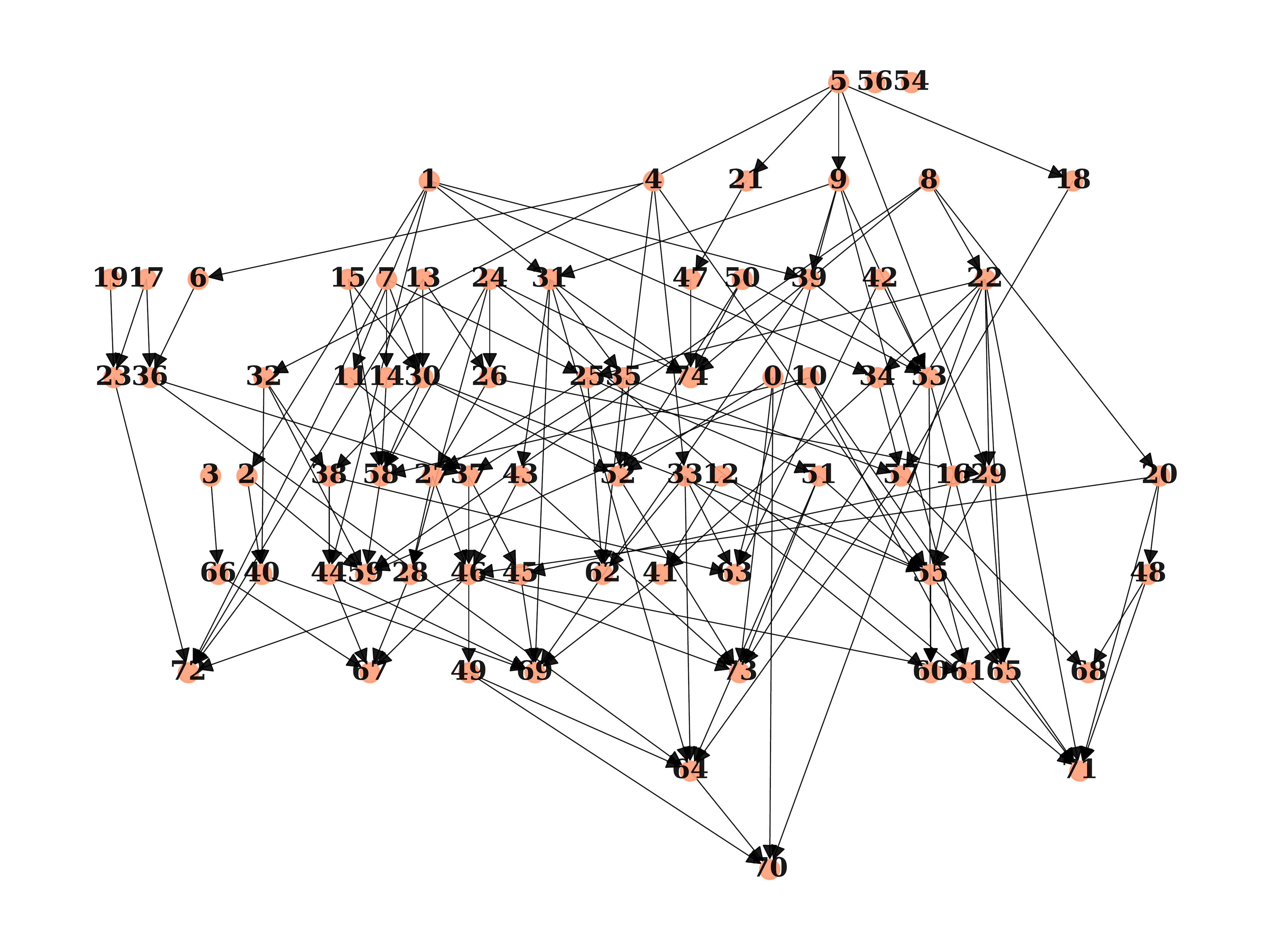}
  \end{subfigure}
\end{figure}

\begin{remark}[Graph Topologies]
  We depict in Figure \ref{fig:random_graph_topologies} the topologies
  of random graphs used in our empirical evaluation.  For values of
  $q$ close to $\frac{2}{n}$, the resulting random graphs tend to have
  a topology which is, at least qualitatively, close to the SCG.  As
  the value of $q$ increases, the random graphs deviate farther from
  the SCG topology, and we therefore expect the LASSO to outperform
  PWGC for larger values of $q$.
\end{remark}

\begin{remark}[MCC as a Support Recovery Measurement]
  We apply Matthew's Correlation Coefficient (MCC)
  \cite{matthews1975comparison} as a statistic for measuring support
  recovery performance (see also \cite{chicco2017ten} tip \# 8).  This
  statistic synthesizes the confusion matrix into a single score
  appropriate for unbalanced labels and is calibrated to fall into the
  range $[-1, 1]$ with $1$ being perfect performance, $0$ being the
  performance of random guessing, and $-1$ being perfectly opposed.
\end{remark}

\begin{remark}[Error Measurement]
  We estimate the 1-step ahead prediction error by forming the variance matrix estimate

  \begin{equation*}
    \widehat{\Sigma}_v \defeq \frac{1}{T_{\text{out}}} \sum_{t = 1}^{T_{\text{out}}} (x(t) - \widehat{x}(t))(x(t) - \widehat{x}(t))^\T
  \end{equation*}

  on a long stream of out-of-sample data.  We then report the quantity

  \begin{equation*}
    \frac{\ln \tr \widehat{\Sigma}_v}{\ln \tr \Sigma_v}
  \end{equation*}

  where $\widehat{\Sigma}_v = \Sigma_v$ is the best possible performance.
\end{remark}

\subsection{Results}
\begin{figure}[!h]
  \centering
  \caption{Simulation Results: PWGC vs AdaLASSO}
  \label{tab:simulation_table}

  \begin{tabular}{|ll||ll|ll|ll|}
    \toprule
    &\textbf{Algorithm}&alasso&pwgc&alasso&pwgc&alasso&pwgc\\
    &\textbf{Metric}&LRE&&FDP&&MCC&\\
    \textbf{T}&\textbf{q}&&&&&&\\
    \midrule
    \multirow{4}{*}{\textbf{50}}
    &\textbf{SCG}&1.71&\textbf{1.55}&0.52&\textbf{0.08}&0.46&\textbf{0.55}\\
    &\textbf{0.04}&1.97&\textbf{1.77}&0.57&\textbf{0.10}&0.41&\textbf{0.53}\\
    &\textbf{0.08}&2.95&\textbf{2.72}&0.50&\textbf{0.23}&0.36&\textbf{0.39}\\
    &\textbf{0.32}&9.02&\textbf{8.17}&\textbf{0.53}&0.56&\textbf{0.14}&0.10\\
    \midrule
    \multirow{4}{*}{\textbf{250}}
    &\textbf{SCG}&1.30&\textbf{1.18}&0.29&\textbf{0.06}&0.70&\textbf{0.81}\\
    &\textbf{0.04}&1.40&\textbf{1.31}&0.30&\textbf{0.07}&0.68&\textbf{0.76}\\
    &\textbf{0.08}&2.49&\textbf{2.21}&0.32&\textbf{0.16}&0.55&\textbf{0.57}\\
    &\textbf{0.32}&8.67&\textbf{7.62}&0.48&0.46&\textbf{0.18}&0.15\\
    \midrule
    \multirow{4}{*}{\textbf{1250}}
    &\textbf{SCG}&1.20&\textbf{1.11}&0.41&\textbf{0.07}&0.68&\textbf{0.88}\\
    &\textbf{0.04}&1.28&\textbf{1.20}&0.46&\textbf{0.07}&0.64&\textbf{0.84}\\
    &\textbf{0.08}&2.12&2.05&0.36&\textbf{0.14}&0.60&\textbf{0.64}\\
    &\textbf{0.32}&7.78&\textbf{7.39}&0.49&\textbf{0.37}&\textbf{0.21}&0.18\\
    \bottomrule
  \end{tabular}

  {\footnotesize Results of Monte Carlo simulations comparing PWGC and
    AdaLASSO $(n = 50, p = 5, p_{\text{max}} = 10)$ for small
    samples and when the SCG assumption doesn't hold.  The superior
    result is bolded when the difference is statistically
    significant, as measured by \texttt{scipy.stats.ttest\_rel}.
    100 iterations are run for each
    set of parameters.\\

    LRE: Log-Relative-Error, i.e. the log sum of squared errors at
    each node relative to the strength of the driving noise
    $\frac{\ln \tr \widehat{\Sigma}_v}{\ln \tr \Sigma_v}$.  FDP: False
    Discovery Proportion.  MCC: Matthew's Correlation Coefficient.\\

    Values of $q$ (edge probability) range between
    $2 / n, 4 / n, 16 / n$ where $2 / n$ has the property that the
    random graphs have on average the same number of edges as the
    SCG.}
\end{figure}

Our simulation results are summarized in Table
\ref{tab:simulation_table}, with additional figures provided in the
supplementary material Section \ref{apx:simulation}.  It is clear that
the superior performance of PWGC in comparison to AdaLASSO is as a
result of limiting the false discovery rate.  It is unsurprising that
PWGC exhibits superior performance when the graph is an SCG, but even
in the case of more general DAGs, the PWGC heuristic is still able to
more reliably uncover the graph structure for small values of $q$.  We
would conjecture that for small $q$, random graphs are ``likely'' to
be ``close'' to SCGs in some appropriate sense.  As $q$ increases,
there are simply not enough edges allowed by the SCG topology for it
to be possible to accurately recover $\gcg$.

Interestingly, we can observe that the AdaLASSO appears to perform
marginally better on strongly causal graphs than directed acyclic
graphs with an equivalent number of edges ($q \approx 0.04$ is chosen for
this purpose).  This provides supporting evidence for one of the main
assertions of this work: that the topological structure of $\gcg$ is
an important distinguishing feature of time series networks in
comparison to classical multivariate regression where it is only the
sparsity rate which is considered.

\section{Conclusion}
\label{sec:conclusion}
In this paper we have argued that considering particular topological
properties of Granger causality networks can provide substantial
insights into the structure of causality graphs with potential for
providing improvements to causality graph estimation when structural
assumptions are met.  In particular, the notion of a strongly-causal
graph has been exploited to establish conditions under which pairwise
causality testing alone is sufficient for recovering a complete
Granger causality graph.  Moreover, examples from the literature
suggest that such topological assumptions may be reasonable in some
applications.  And secondly, even when the strong-causality assumption
is not met, we have provided simulation evidence to suggest that our
pairwise testing algorithm PWGC can still outperform the LASSO and
adaLASSO, both of which are commonly employed in applications.

We emphasize that the causality graph topology is one of the key
defining features of time series analysis in comparison to standard
multivariate regression and therefore advocate for further study of
how different topological assumptions may impact the recovery of
causality graphs.  For example, are there provable guarantees on the
error rate of PWGC when applied to non strongly-causal graphs?  Can
constraint systems or cunning adaptive weighting schemes impose useful
prior knowledge about graph topology for the LASSO algorithm?
Finally, the work of \cite{barnett2015granger} has established the
superiority of Granger causality testing by state space models (as
opposed to pure autoregressions) in many cases.  Combining this work
with our PWGC algorithm (by modifying the approach described in
Section \ref{sec:pairwise_hypothesis_testing} to instead utilize
state-space Granger causality testing) therefore is likely to enable
application to very large networks of time series data which are not
well approximated by finite $\VAR(p)$ models.  Moreover, our
heuristics are in principle applicable in a model-free context.  As
long as a primitive for testing the pairwise causation between two
components of a multivariate time series is available, our methods may
be useful.

\appendix

\begin{appendices}
  \numberwithin{equation}{section}


\section{Overview}
We restate our main results and provide detailed proofs.  Simple
Corollaries have their proofs in the main text, and are occasionally
referenced here.  The main Theorem is proven in Section
\ref{apx:proof_main_theorem}, and all of the building blocks are
established in Section \ref{apx:ancillary_results}.

We detail the methods used for our simulations and finite sample
implementation in Section \ref{sec:structure_learning} and provide
additional simulation results in Section \ref{apx:simulation}.

References to equations, lemmas, etc., in this document are prefixed
with their section, whereas prefix-free equation numbers refer to the
main document.  e.g. ``Equation (1)'' refers to the first equation in
the main document, and ``Equation (A.1)'' refers to the first equation
in this document.

Code will be made available at
\url{https://github.com/RJTK/granger_causality}, as well as
accompanying this supplementary material.

\section{Proofs}
\subsection{Preparatory Results}
\label{apx:ancillary_results}
\begin{theorem}[Granger Causality Equivalences \ref{thm:granger_causality_equivalences}]
  The following are equivalent:

  \begin{enumerate}
  \item{$x_j \ngc x_i$}
  \item{$\forall \tau \in \N_+\ B_{ij}(\tau) = 0$ i.e. $\B_{ij}(z) = 0$}
  \item{$H_t^{(i)} \perp \H_{t - 1}^{(j)}\ |\ \H_{t - 1}^{(-j)}$}
  \item{$\linE{x_i(t)}{\H_{t - 1}^{(-j)}} = \linE{x_i(t)}{\H_{t - 1}}$}
  \end{enumerate}
\end{theorem}

\begin{proof}
  $(a) \Rightarrow (b)$ follows as a result of the uniqueness of orthogonal
  projection (i.e. the best estimate is necessarily the coefficients
  of the model).  $(b) \Rightarrow (c)$ follows since in computing
  $(y - \linE{y}{\H_{t - 1}^{(-j)}})$ for $y \in H_t^{(i)}$ it is sufficient
  to consider $y = x_i(t)$ by linearity, then since
  $H_{t - 1}^{(i)} \subseteq \H_{t - 1}^{(-j)}$ we have
  $(x_i(t) - \linE{x_i(t)}{\H_{t - 1}^{(-j)}}) = v_i(t)$ since
  $\B_{ij}(z) = 0$.  The result follows since
  $v_i(t) \perp \H_{t - 1}$.  $(c) \iff (d)$ is a result of the equivalence
  in Definition \ref{lem:conditional_orthogonality_equivalence}.  And,
  $(d) \implies (a)$ follows directly from the Definition.
\end{proof}

\begin{lemma}
  \label{lem:adj_matrix}
  Let $S$ be the transposed adjacency matrix\footnote{\footnotesize We
    are using the convention that $\B_{ij}(z)$ is a filter with input
    $x_j$ and output $x_i$ so as to write the action of the system as
    $\B(z)x(t)$ with $x(t)$ as a column vector.  This competes with
    the usual convention for adjacency matrices where $A_{ij} = 1$ if
    there is an edge $(i, j)$.  In our case, the sparsity pattern of
    $\B_{ij}$ is the \textit{transposed} conventional adjacency
    matrix.} of the Granger causality graph $\gcg$.  Then,
  $(S^k)_{ij}$ is the number of paths of length $k$ from node $j$ to
  node $i$.  Evidently, if $\forall k \in \N,\ (S^k)_{ij} = 0$ then
  $j \not\in \anc{i}$.
\end{lemma}
\begin{proof}
  This is a well known theorem, proof follows by induction.
\end{proof}

\begin{proposition}[Ancestor Expansion]
  \label{prop:parent_expanding}
  The component $x_i(t)$ of $x(t)$ can be represented in terms of it's
  parents in $\gcg$:

  \begin{equation}
    \label{eqn:parent_expansion}
    x_i(t) = v_i(t) + \B_{ii}(z)x_i(t) + \sum_{k \in \pa{i}}\B_{ik}(z)x_k(t).
  \end{equation}

  Moreover, $x_i$ can be expanded in terms of it's ancestor's $v(t)$
  components only:

  \begin{equation}
    \label{eqn:ancestor_expansion}
    x_i(t) = \A_{ii}(z)v_i(t) + \sum_{\substack{k \in \anc{i} \\ k \ne i}}\A_{ik}(z)v_k(t),
  \end{equation}

  where $\A(z) = \sum_{\tau = 0}^\infty A(\tau)z^{-\tau}$ is the filter from
  the Wold decomposition representation of $x(t)$, Equation
  \eqref{eqn:wold}.
\end{proposition}

\begin{proof}
  Equation \eqref{eqn:parent_expansion} is immediate from the
  $\VAR(\infty)$ representation of Equation \eqref{eqn:ar_representation}
  and Theorem \ref{thm:granger_causality_equivalences}, we are left to
  demonstrate \eqref{eqn:ancestor_expansion}.
  
  From Equation \eqref{eqn:ar_representation}, which we are assuming
  throughout the paper to be invertible, we can write

  \begin{equation*}
    x(t) = (I - \B(z))^{-1} v(t),
  \end{equation*}

  where necessarily $(I - \B(z))^{-1} = \A(z)$ due to the uniqueness
  of the Wold decomposition.  Since $\B(z)$ is stable we have

  \begin{equation}
    \label{eqn:resolvant_inv}
    (I - \B(z))^{-1} = \sum_{k = 0}^\infty \B(z)^k.
  \end{equation}

  Invoking the Cayley-Hamilton theorem allows writing the infinite sum
  of \eqref{eqn:resolvant_inv} in terms of \textit{finite} powers of
  $\B$.

  Let $S$ be a matrix with elements in $\{0, 1\}$ which represents the
  sparsity pattern of $\B(z)$, from Lemma \ref{lem:adj_matrix} $S$ is
  the transpose of the adjacency matrix for $\gcg$ and hence
  $(S^k)_{ij}$ is non-zero if and only if $j \in \gpn{k}{i}$, and
  therefore $\B(z)^k_{ij} = 0$ if $j \not \in \gpn{k}{i}$.  Finally,
  since $\anc{i} = \bigcup_{k = 1}^n\gpn{k}{i}$ we see that
  $\A_{ij}(z)$ is zero if $j \not\in \anc{i}$.

  Therefore

  \begin{align*}
    x_i(t) &= [(I - \B(z))^{-1}v(t)]_i\\
           &= \sum_{j = 1}^n \A_{ij}(z) v_j(t)\\
           &= \A_{ii}(z) v_i(t) + \sum_{\substack{j \in \anc{i} \\ j \ne i}} \A_{ij}(z) v_j(t)
  \end{align*}
\end{proof}

\begin{proposition}
  \label{prop:separated_ancestor_uncorrelated}
  Consider distinct nodes $i, j$ in a Granger causality graph
  $\gcg$.  If

  \begin{enumerate}[label=(\alph*)]
  \item{$j \not\in \anc{i}$ and $i \not\in \anc{j}$}
  \item{$\anc{i}\cap\anc{j} = \emptyset$}
  \end{enumerate}

  then $\H_t^{(i)} \perp \H_t^{(j)}$, that is,
  $\forall s, \tau \in \Z_+\ \E[x_i(t - s)x_j(t - \tau)] = 0$.  Moreover,
  this means that $j \npwgc i$ and $\linE{x_j(t)}{\H_t^{(i)}} = 0$.
\end{proposition}

\begin{proof}
  We show directly that
  $\forall s, \tau \in \Z_+\ \E[x_i(t - s)x_j(t - \tau)] = 0$.  To this end, fix
  $s, \tau \ge 0$, then by expanding with Equation
  \eqref{eqn:ancestor_expansion} we have

  \begin{align*}
    \E x_i(t - s)x_j(t - \tau)
    &= \E \big(\A_{ii}(z)v_i(t - s)\big)\big(\A_{jj}(z)v_j(t - \tau)\big)\\
    &+ \sum_{\substack{k \in \anc{i} \\ k \ne i}}\E[\big(\A_{ik}(z)v_k(t - s)\big)\big(\A_{jj}(z)v_j(t - \tau)\big)]\\
    &+ \sum_{\substack{\ell \in \anc{j} \\ \ell \ne j}}\E[\big(\A_{ii}(z)v_i(t - s)\big) \big(\A_{j\ell}(z) v_\ell(t - \tau)\big)]\\
    &+ \sum_{\substack{k \in \anc{i} \\ k \ne i}}\sum_{\substack{\ell \in \anc{j} \\ \ell \ne j}}\E[\big(\A_{ik}(z)v_k(t - s)\big)\big(\A_{j\ell}(z)v_\ell(t - \tau)\big)].
  \end{align*}
  
  Keeping in mind that $v(t)$ is an isotropic and uncorrelated
  sequence we see that each of these above four terms are 0: the
  first term since $i \ne j$, the second and third since
  $j \not\in \anc{i}$ and $i \not\in \anc{j}$ and finally the fourth since
  $\anc{i} \cap \anc{j} = \emptyset$.
\end{proof}

\begin{lemma}
  \label{lem:vj_perp}
  Consider distinct nodes $i, j$ in a Granger causality graph $\gcg$.
  If $j \not \in \anc{i}$, then $\H_t^{(v_j)} \perp \H_{t}^{(i)}$, and therefore
  for any causal filter $\Phi(z)$ we have

  \begin{align*}
    \linE{\Phi(z)v_j(t)}{\H_{t - 1}^{(i)}} &= 0\\
    \inner{x_i(t)}{\Phi(z)v_j(t)} & = 0.
  \end{align*}
\end{lemma}
\begin{proof}
  Fix $\tau, s \ge 0$, then by expanding with Equation \eqref{eqn:ancestor_expansion}

  \begin{align*}
    \E[x_i(t - \tau)v_j(t - s)] &= \E[\big(\A_{ii}(z)v_i(t) + \sum_{k \in \anc{i}}\A_{ik}(z)v_k(t) \big) v_j(t - s)]\\
    &= 0.
  \end{align*}

  This follows since $i \ne j$ and $j \not \in \anc{i}$ and $v(t)$ is
  isotrophic and uncorrelated.
\end{proof}

\begin{proposition}
  \label{prop:ancestor_uncorrelated}
  Consider distinct nodes $i, j$ in a Granger causality graph $\gcg$.
  If

  \begin{enumerate}[label=(\alph*)]
  \item{$j \not\in \anc{i}$}
  \item{$\anc{i}\cap\anc{j} = \emptyset$}
  \end{enumerate}

  then $j \npwgc i$.
\end{proposition}
\begin{proof}
  By Theorem \ref{thm:granger_causality_equivalences} it suffices to show that

  \begin{equation*}
    \forall \psi \in \H_{t - 1}^{(j)}\ \inner{x_i(t) - \linE{x_i(t)}{\H_{t - 1}^{(i)}}}{\psi - \linE{\psi}{\H_{t - 1}^{(i)}}} = 0.
  \end{equation*}

  which by the orthogonality principle and by representing
  $\psi \in \H_{t - 1}^{(j)}$ via the action of some strictly causal filter
  $\Phi(z)$ on $x_j(t)$ is equivalent to

  \begin{equation}
    \label{eqn:proof_inner0}
    \inner{x_i(t)}{\Phi(z)x_j(t) - \linE{\Phi(z)x_j(t)}{\H_{t - 1}^{(i)}}} = 0.
  \end{equation}

  If we expand $x_j(t)$ using Equation \eqref{eqn:ancestor_expansion},
  the left hand side of \eqref{eqn:proof_inner0} becomes

  \begin{equation*}
    \inner{x_i(t)}{\sum_{k \in \anc{j} \cup \{j\}} \Big(\Phi(z)\A_{jk}(z)v_k(t) - \linE{\Phi(z)\A_{jk}(z)v_k(t)}{\H_{t - 1}^{(i)}}\Big)}.
  \end{equation*}

  We see that this is $0$ by Lemma \ref{lem:vj_perp} since
  $j \not \in \anc{i}$, and
  \[
    \anc{i} \cap \anc{j} = \emptyset \implies \forall k \in \anc{j}: k \not \in \anc{i}.
  \]
\end{proof}

\begin{remark}
  In order to prove Proposition \ref{prop:ancestor_properties} we
  require some additional notation, as well as another representation
  theorem.  The difficulty addressed by the following Definition
  \ref{def:scc} and Lemma \ref{lem:scc_expansion} is that in the
  representation of $x_j(t)$ in terms of it's parents (i.e. Equation
  \eqref{eqn:parent_expansion})

  \[
    x_i(t) = v_i(t) + \B_{ii}(z)x_i(t) + \sum_{k \in \pa{i}}\B_{ik}(z)x_k(t),
  \]

  the filter $\B_{ii}(z)$ need not be stable.  That is, the inverse
  filter $(1 - \B_{ii}(z))^{-1}$ need not exist.  An example of this
  issue is furnished by

  \[
    \B(z) = \left[
    \begin{array}{cc}
      \rho & -a\\
      a & 0\\
    \end{array}
  \right] z^{-1},
  \]

  for which, depending on the value of $a$, may still be stable even
  if $|\rho| > 1$.  This implies that it is not always possible to
  represent $x_i(t)$ in terms of $v_i(t)$ and $x_k(t), k \in \pa{i}$
  alone, i.e. as

  \[
    x_i(t) = (1 - \B_{ii}(z))^{-1}\big(v_i(t) + \sum_{k \in \pa{i}}\B_{ik}(z)x_k(t)\big).
  \]

  The difficulty presented by the non-existence of such a
  representation may become apparent upon studying the proof of
  Proposition \ref{prop:ancestor_properties}.
\end{remark}

\begin{definition}[Strongly Connected Components]
  \label{def:scc}
  In a graph $\gcg$, the \textit{ordered} (by the natural ordering on
  $\N$) subset $S \subseteq [n]$ is \textit{strongly connected} if
  $\forall i, j \in S$, $i \in \anc{j}$ and $j \in \anc{i}$.  We will
  denote by $S(j)$ (which may be the singleton $(j)$) the largest
  strongly connected component (SCC) containing $j$.  We will denote
  $x_{S(j)}(t)$ to be the vector of processes

  \[
    x_{S(j)}(t) = \big(x_s(t)\ |\ s \in S(j)\big),
  \]

  whose indices are given the same (natural) ordering as $S(j)$.
  Similarly, the sub-filter of $\B(z)$ acting on $x_{S(j)}(t)$ will be
  denoted $\B_{S(j)}(z)$.
\end{definition}

\begin{lemma}[Expansion in SCCs]
  \label{lem:scc_expansion}
  Given some $j \in [n]$, the process $x_{S(j)}(t)$ can be represented
  by

  \begin{equation}
    \label{eqn:scc_parent_expansion}
    x_{S(j)}(t) = v_{S(j)}(t) + \B_{S(j)}(z)x_{S(j)}(t) + \sum_{\substack{s \in S(j) \\ k \in \pa{s} \cap S(j)^\c}} B_{sk}(z)x_k(t)e_s^{S(j)},
  \end{equation}

  where $e_s^{S(j)}$ denotes the length $|S(j)|$ canonical basis
  vector with a $1$ in the component corresponding to $x_s$ in the
  vector $x_{S(j)}$, and the summation is a double sum on $s$ and $k$.

  Moreover, the filter $\B(z)$ is stable with $I - \B_{S(j)}(z)$
  invertible:

  \begin{equation}
    \label{eqn:scc_inversion}
    (I - \B_{S(j)}(z))^{-1} = \sum_{k = 0}^\infty \B_{S(j)}(z)^k,
  \end{equation}

  therefore

  \begin{equation}
    \label{eqn:scc_parent_expansion_inverted}
    x_{S(j)}(t) = (I - \B_{S(j)}(z))^{-1} \big(v_{S(j)}(t) + \sum_{\substack{s \in S(j) \\ k \in \pa{s} \cap S(j)^\c}} B_{sk}(z)x_k(t)e_s^{S(j)}\big).
  \end{equation}
  
\end{lemma}
\begin{proof}
  The representation of Equation \eqref{eqn:scc_parent_expansion}
  follows directly from the $\VAR$ representation of $x(t)$ (i.e. Equation
  \eqref{eqn:ar_representation})

  \[
    x(t) = \B(z)x(t) + v(t),
  \]

  which, when rearranged appropriately, can be written as

  \[
    \left[
      \begin{array}{c}
        x_{S(j)}(t)\\
        x_{S(j)^c}(t)
      \end{array}
    \right] =
    \left[
      \begin{array}{cc}
        \B_{S(j)}(z)& \B_{S(j), S(j)^\c}(z)\\
        \B_{S(j)^\c, S(j)}(z)& \B_{S(j)^\c}(z)
      \end{array}
    \right]
        \left[
      \begin{array}{c}
        x_{S(j)}(t)\\
        x_{S(j)^c}(t)
      \end{array}
    \right] +
    \left[
      \begin{array}{c}
        v_{S(j)}(t)\\
        v_{S(j)^c}(t)
      \end{array}
    \right].
  \]

  Theorem \ref{thm:granger_causality_equivalences} is invoked in order
  to restrict the summation to $k \in \pa{s}$ (since other elements
  are $0$).

  Now, we can partition $\gcg$ into it's maximal SCCs
  $S_1, \ldots, S_N$, (one of which is $S(j)$) and then consider the
  DAG formed on $N$ nodes with edges $I \rightarrow J$ included on the
  condition that
  $\exists j \in S_J, i \in S_I \text{ s.t. } i \in \anc{j}$.  By
  topologically sorting this DAG, we obtain an ordering $\sigma$ of
  $[n]$ such that $\B_\sigma(z)$ is block upper triangluar, with one
  of it's diagonal blocks consisting of the (possibly reordered)
  matrix $\B_{S(j)}(z)$.  So we have

  \begin{align*}
    \forall\ |z^{-1}| \le 1: \det \B(z) = \prod_{i = 1}^N \det \B_{S_i}(z) &\ne 0\\
    \implies \forall\ |z^{-1}| \le 1: \det \B_{S(j)}(z) &\ne 0,
  \end{align*}

  and therefore $\B_{S(j)}(z)$ is stable, invertible, and Equation
  \eqref{eqn:scc_inversion} holds.
\end{proof}

\begin{proposition}
  \label{prop:ancestor_properties}
  If in a Granger causality graph $\gcg$ where $j \pwgc i$ then
  $j \in \anc{i}$ or $\exists k \in \anc{i} \cap\anc{j}$ which is a
  confounder of $(i, j)$.
\end{proposition}

\begin{proof}
  We will prove by way of contradiction.  To this end, suppose that
  $j$ is a node such that:

  \begin{enumerate}[label=(\alph*)]
    \item{$j \not \in \anc{i}$}
    \item{every $k \in \anc{i} \cap \anc{j}$ every
        $k \rightarrow \cdots \rightarrow j$ path contains $i$.}
  \end{enumerate}

  Firstly, notice that every $u \in \big(\pa{j} \setminus \{i\}\big)$
  necessarily inherits these same two properties.  This follows since
  if we also had $u \in \anc{i}$ then $u \in \anc{i} \cap \anc{j}$ so by our
  assumption every $u \rightarrow \cdots \rightarrow j$ path must contain
  $i$, but $u \in \pa{j}$ so $u \rightarrow j$ is a path that doesn't contain
  $i$, and therefore $u \not\in \anc{i}$; moreover, if we consider
  $w \in \anc{i} \cap \anc{u}$ then we also have
  $w \in \anc{i} \cap \anc{j}$ so the assumption implies that every
  $w \rightarrow \cdots \rightarrow j$ path must contain $i$.  These properties therefore
  extend inductively to every $u \in \big(\anc{j} \setminus \{i\}\big)$.

  In order to deploy a recursive argument, define the following
  partition of $\pa{u}$, for some node $u$:

  \begin{align*}
    C_0(u) &= \{k \in \pa{u}\ |\ i \not\in \anc{k}, \anc{i} \cap \anc{k} = \emptyset, k \ne i\}\\
    C_1(u) &= \{k \in \pa{u}\ |\ i \in \anc{k} \text{ or } k = i\}\\
    C_2(u) &= \{k \in \pa{u}\ |\ i \not\in \anc{k}, \anc{i} \cap \anc{k} \ne \emptyset, k \ne i\}.
  \end{align*}

  We notice that for any $u$ having the properties $(a), (b)$ above,
  we must have $C_2(u) = \emptyset$ since if $k \in C_2(u)$ then
  $\exists w \in \anc{i} \cap \anc{k}$ (and
  $w \in \anc{i} \cap \anc{u}$, since $k \in \pa{u}$) such that
  $i \not \in \anc{k}$ and therefore there must be a path
  $\gcgpath{w}{k} \rightarrow u$ which does not contain $i$,
  contradicting property $(b)$.  Moreover, for any
  $u \in \big(\anc{j} \setminus \{i\}\big)$ and $k \in C_0(u)$,
  Proposition \ref{prop:ancestor_uncorrelated} shows that
  $H_t^{(i)} \perp \H_{t - 1}^{(j)} | \H_{t - 1}^{(i)}$.

  In order to establish $j \npwgc i$, choose an arbitrary element of
  $\H_{t - 1}^{(j)}$ and represent it via the action of a strictly causal
  filter $\Phi(z)$, i.e.  $\Phi(z) x_j(t) \in \H_{t - 1}^{(j)}$, by
  Theorem \ref{thm:granger_causality_equivalences} it suffices to show
  that

  \begin{equation}
    \label{eqn:sufficient_inner_prod}
    \inner{x_i(t)}{\Phi(z)x_j(t) - \linE{\Phi(z)x_j(t)}{\H_{t - 1}^{(i)}}} = 0.
  \end{equation}

  Denote $e_j \defeq e_j^{S(j)}$, we can write
  $x_j(t) = e_j^\T x_{S(j)}(t)$, and therefore from Equation
  \eqref{eqn:scc_parent_expansion_inverted} there exist strictly
  causal filters $\Gamma_s(z)$ and $\Lambda_{sk}(z)$ (defined for ease
  of notation) such that

  \[
    x_j(t) = \sum_{s \in S(j)} \Gamma_s(z)v_s(t) + \sum_{\substack{s \in S(j) \\ k \in \pa{s} \cap S(j)^\c}} \Lambda_{sk}(z)x_k(t).
  \]

  When we substitute this expression into the left hand side of
  Equation \eqref{eqn:sufficient_inner_prod}, we may cancel each term
  involving $v_s$ by Lemma \ref{lem:vj_perp}, and each $k \in C_0(s)$
  by our earlier argument, leaving us with

  \[
    \sum_{\substack{s \in S(j) \\ k \in C_1(s) \cap S(j)^\c}}\inner{x_i(t)}{\Phi(z)\Lambda_{sk}(z)x_k(t) - \linE{\Phi(z)\Lambda_{sk}(z)x_k(t)}{\H_{t - 1}^{(i)}}}.
  \]

  Since each $k \in C_1(s)$ with $k \ne i$ inherits properties $(a)$
  and $(b)$ above, we can recursively expand each $x_k$ of the above
  summation until reaching $k = i$ (which is garaunteed to terminate
  due to the definition of $C_1(u)$) which leaves us with some
  strictly causal filter $F(z)$ such that the left hand side of
  Equation \eqref{eqn:sufficient_inner_prod} is equal to

  \[
    \inner{x_i(t)}{\Phi(z)F(z)x_i(t) - \linE{\Phi(z)F(z)x_i(t)}{\H_{t - 1}^{(i)}}},
  \]

  and this is $0$ since $\Phi(z)F(z)x_i(t) \in \H_{t - 1}^{(i)}$.
\end{proof}

\begin{proposition}
  \label{prop:sc_graph_common_anc}
  In a strongly causal graph if $j \in \anc{i}$ then any
  $k \in \anc{i} \cap \anc{j}$ is not a confounder, that is,
  the unique path from $k$ to $i$ contains $j$.
\end{proposition}
\begin{proof}
  Suppose that there is a path from $k$ to $i$ which does not contain
  $j$.  In this case, there are multiple paths from $k$ to $i$ (one of
  which \textit{does} go through $j$, since $j \in \anc{i}$) which
  contradicts the assumption of strong causality.
\end{proof}

\begin{proposition}
  \label{prop:pwgc_anc}
  If $\gcg$ is a strongly causal DAG then $j \in \anc{i} \Rightarrow j \pwgc i$.
\end{proposition}
\begin{proof}
  We will show that for some $\psi \in \H_{t - 1}^{(j)}$ we have

  \begin{equation}
    \label{eqn:cond_ortho_proof}
    \inner{\psi - \linE{\psi}{\H_{t - 1}^{(i)}}}{x_i(t) - \linE{x_i(t)}{\H_{t - 1}^{(i)}}} \ne 0
  \end{equation}

  and therefore that $H_t^{(i)} \not\perp\ \H_{t - 1}^{(j)}\ |\ \H_{t - 1}^{(i)}$, which by Theorem (\ref{thm:granger_causality_equivalences}) is enough to establish that $j \pwgc i$.

  Firstly, we will establish a representation of $x_i(t)$ that involves $x_j(t)$.  Denote by $a_{r + 1} \rightarrow a_r \rightarrow \cdots \rightarrow a_1 \rightarrow a_0$ with $a_{r + 1} \defeq j$ and $a_0 \defeq i$ the \textit{unique} $\gcgpath{j}{i}$ path in $\gcg$, we will expand the representation of Equation \eqref{eqn:parent_expansion} backwards along this path:

  \begin{align*}
    x_i(t) &= v_i(t) + \B_{ii}(z) x_i(t) + \sum_{k \in \pa{i}}\B_{ik}(z) x_k(t)\\
           &= \underbrace{v_{a_0}(t) + \B_{a_0a_0}(z) x_i(t) + \sum_{\substack{k \in \pa{a_0} \\ k \ne a_1}}\B_{a_0 k}(z) x_k(t)}_{\defeq \wtalpha{a_0}{a_1}} + \B_{a_0a_1}(z)x_{a_1}(t)\\
           &= \wtalpha{a_0}{a_1} + \B_{a_0a_1}(z)\big[\wtalpha{a_1}{a_2} + \B_{a_1a_2}(z)x_{a_2}(t) \big]\\
           &\overset{(a)}{=} \sum_{\ell = 0}^r \underbrace{\Big(\prod_{m = 0}^{\ell - 1} \B_{a_m a_{m + 1}}(z) \Big)}_{\defeq F_\ell(z)} \wtalpha{a_\ell}{a_{\ell + 1}} + \Big(\prod_{m = 0}^{r}\B_{a_m a_{m + 1}}(z)\Big)x_{a_{r + 1}}(t)\\
           &= \sum_{\ell = 0}^r F_\ell(z) \wtalpha{a_\ell}{a_{\ell + 1}} + F_{r + 1}(z) x_j(t)
  \end{align*}

  where $(a)$ follows by a routine induction argument and where we define $\prod_{m = 0}^{-1} \bullet \defeq 1$ for notational convenience.

  Using this representation to expand Equation \eqref{eqn:cond_ortho_proof}, we obtain the following cumbersome expression:

  \begin{align*}
    &\inner{\psi - \linE{\psi}{\H_{t - 1}^{(i)}}}{F_{r + 1}(z)x_j(t) - \linE{F_{r + 1}(z)x_j(t)}{\H_{t - 1}^{(i)}}}\\
    &- \inner{\psi - \linE{\psi}{\H_{t - 1}^{(i)}}}{\linE{\sum_{\ell = 0}^r F_\ell(z)\wtalpha{a_\ell}{a_{\ell + 1}}}{\H_{t - 1}^{(i)}}}\\
    &+ \inner{\psi - \linE{\psi}{\H_{t - 1}^{(i)}}}{\sum_{\ell = 0}^r F_\ell(z)\wtalpha{a_\ell}{a_{\ell + 1}}}.
  \end{align*}

  Note that by the orthogonality principle, $\psi - \linE{\psi}{\H_{t - 1}^{(i)}} \perp \H_{t - 1}^{(i)}$, the middle term above is $0$.  Choosing now the particular value $\psi = F_{r + 1}(z)x_j(t) \in \H_{t - 1}^{(j)}$ we arrive at

  \begin{align*}
    &\inner{\psi - \linE{\psi}{\H_{t - 1}^{(i)}}}{x_i(t) - \linE{x_i(t)}{\H_{t - 1}^{(i)}}}\\
    &= \E|F_{r + 1}(z)x_j(t) - \linE{F_{r + 1}(z)x_j(t)}{\H_{t - 1}^{(i)}}|^2\\
    &+ \inner{F_{r + 1}(z)x_j(t) - \linE{F_{r + 1}(z)x_j(t)}{\H_{t - 1}^{(i)}}}{\sum_{\ell = 0}^r F_\ell(z) \wtalpha{a_\ell}{a_{\ell + 1}}}.
  \end{align*}

  Now since $F_{r + 1}(z) \ne 0$ by Theorem
  \ref{thm:granger_causality_equivalences}, and
  $F_{r + 1}(z)x_j(t) \not\in \H_{t - 1}^{(i)}$, we have by the
  Cauchy-Schwarz inequality that this expression is equal to $0$ if
  and only if

  \begin{equation*}
    \sum_{\ell = 0}^r F_\ell(z) \wtalpha{a_\ell}{a_{\ell + 1}} \overset{\text{a.s.}}{=} \linE{F_{r + 1}(z)x_j(t)}{\H_{t - 1}^{(i)}} - F_{r + 1}(z)x_j(t),
  \end{equation*}

  or by rearranging and applying the representation for $x_i(t)$
  obtained earlier, if and only if

  \begin{equation*}
    x_i(t) \overset{\text{a.s.}}{=} \linE{F_{r + 1}(z)x_j(t)}{\H_{t - 1}^{(i)}}.
  \end{equation*}

  But, this is impossible since $x_i(t) \not \in \H_{t - 1}^{(i)}$.
\end{proof}

\begin{example}
  Consider a process $x(t)$ generated by the $\VAR(1)$
  model\footnote{Recall that any $\VAR(p)$ model with $p < \infty$ can
    be written as a $\VAR(1)$ model, so we lose little generality in
    considering this case.}  having $\B(z) = Bz^{-1}$.  If $B$ is
  diagonalizable, and has at least $2$ distinct eigenvalues, then
  $x(t)$ is persistent.

  Pick any $i \in [n], j \in \anc{i} \setminus \{i\}$.  Then the
  stability of $B$ allows us to write

  \begin{equation*}
    \A(z) = \sum_{k = 0}^\infty B^k z^{-k},
  \end{equation*}

  whereby we see that $\exists k > 0$ such that $[B^k]_{ij} \ne 0$
  (since $j \in \anc{i}$).  Then consider

  \begin{equation*}
    \begin{aligned}
      [B^{rk}]_{ij} &= e_i^\T B^{rk} e_j \\
      &\overset{(a)}{=} \big((P^\T e_i)^\T J^{rk} P^{-1}e_j\big)\\
      &= \tr [(P^\T e_i)^\T J^{rk} P^{-1}e_j]\\
      &\overset{(b)}{=} \tr [(J^{rk}) (v u^\T)],
    \end{aligned}
  \end{equation*}

  where $(a)$ utilizes the Jordan Normal Form of $B$, and $(b)$
  denotes $u = P^\T e_i$ and $v = P^{-1}e_j$.  In order for
  $\tau_\infty(\A_{ij}) < \infty$, there must be some $N > 1$ such
  that $\forall r \ge N$, the above term is $0$.  This may be the case
  for instance if $B$ is a nilpotent matrix.  

  Using the supposition that $B$ is diagonalizable (i.e. $J$ is a
  diagonal matrix) with at least $2$ distinct eigenvalues (in this
  case $B$ is \textit{not} nilpotent), we can then rewrite the above
  as

  \begin{equation*}
    f(r) \defeq \tr [(J^{rk}) (v u^\T)] = \sum_{\nu = 1}^n \lambda_\nu ^{rk} v_\nu u_\nu \defeq \sum_{\nu = 1}^n \lambda_\nu^{rk} \beta_\nu
  \end{equation*}

  where $\lambda_\nu$ denotes the eigenvalues of $B$ and
  $\beta_\nu = u_\nu v_\nu$.  Note that $f(0) = 0$ since $i \ne j$ and
  $u$ is a row of $P$ and $v$ is a column of $P^{-1}$.  Moreover,
  $f(1) \ne 0$ by hypothesis.  But, in order for
  $f(r) = 0\ \forall r \ge N$, it would need to be the case that

  \begin{equation*}
    \Dg(\bm{\lambda})^r \bm{\lambda} = Vz
  \end{equation*}

  had a solution in $z$ for every $r \ge N$, where $V$ is an
  $n \times n - 1$ full-rank matrix whose columns span the nullspace of $\beta$,
  and $\bm{\lambda} = (\lambda_1, \ldots, \lambda_n)$. That is,
  iterates of $\Dg(\bm{\lambda})$ applied to $\bm{\lambda}$ would need to remain
  inside $\beta$'s nullspace.  This would imply that

  \begin{equation*}
    VV^\dagger \bm{\lambda}^{r + 1} = \bm{\lambda}^{r + 1},
  \end{equation*}

  i.e. that $\bm{\lambda}^{r + 1}$ is an eigenvector of $VV^\dagger$
  for an infinite number of integers $r$ (the exponentiation is to be
  understood as a point wise operation).  However, since there can only
  be a finite number of (unit length) eigenvectors, this cannot be the
  case unless every eigenvalue $(\lambda_1, \ldots, \lambda_n)$ were
  equal.

  We see from this example that the collection of $\VAR(1)$ systems
  which are not persistent are pathological, in the sense that their
  system matrices have zero measure when viewed as a subset of $\R^{n^2}$.
\end{example}

\begin{lemma}
  \label{lem:time_lag_cancellation}
  Suppose $v(t)$ is a scalar process with unit variance and zero
  autocorrelation and let $\A(z), \B(z)$ be nonzero and strictly
  causal (i.e. $1 \le \tau_0(\A) < \infty$,
  $1 \le \tau_0(\B) < \infty$) linear filters.  Then,

  \begin{equation}
    \inner{F(z)\A(z)v(t)}{\B(z)v(t)} = 0\ \forall \text{ strictly causal filters } F(z)
  \end{equation}

  if and only if $\tau_0(\A) \ge \tau_\infty(\B)$.
\end{lemma}
\begin{proof}
  We have

  \begin{align}
    \inner{\A(z)v(t)}{\B(z)v(t)} &= \sum_{\tau = 1}^\infty \sum_{s = 1}^\infty a(\tau)b(s)\E[v(t - s)v(t - \tau)]\\
    &= \sum_{\tau = \text{max}(\tau_0(\A), \tau_0(\B))}^{\text{min}(\tau_\infty(\A), \tau_\infty(\B))} a(\tau) b(\tau),\\
  \end{align}

  since $\E[v(t - s)v(t - \tau)] = \delta_{s - \tau}$.  This expression is
  $0$ if and only if $\tau_0(\A) \ge 1 + \tau_\infty(\B)$ or if
  $\tau_0(\B) \ge 1 + \tau_\infty(\A)$ or if the coefficients are orthogonal along
  the common support.

  Specializing this fact to $\inner{F(z)\A(z)v(t)}{\B(z)v(t)}$ we
  see that the coefficients cannot be orthogonal for every choice of
  $F$, and that $\text{sup}_F \tau_\infty(F\A) = \infty$, leaving only
  the possibility that

  \begin{align*}
     \forall F\ \tau_0(F\A) \ge 1 + \tau_\infty(\B) &\overset{(a)}{\iff} \tau_0(\A) \ge 1 + \tau_\infty(\B) - \underset{F}{\text{min }} \tau_0(F)\\
    &\overset{(b)}{\iff} \tau_0(\A) \ge \tau_\infty(\B),
  \end{align*}

  where $(a)$ follows since $\tau_0(F\A) = \tau_0(F) + \tau_0(\A)$,
  and $(b)$ since $\text{min}_F\ \tau_0(F) = 1$.
\end{proof}

\begin{corollary}
  \label{cor:time_lag_cancellation}
  For $k \in \anc{i} \cap \anc{j}$ we have

  \begin{align*}
    \linE{F(z)\A_{jk}(z)v_k(t)}{\H_{t - 1}^{(i)}} &= 0\ \forall \text{ strictly causal } F(z)\\
    \iff \inner{F(z)\A_{jk}(z)v_k(t)}{\A_{ik}(z)v_k(t)} &= 0\ \forall \text{ strictly causal } F(z)\\
    \iff \tau_0(\A_{jk}) \ge \tau_\infty(\A_{ik})
  \end{align*}
\end{corollary}
\begin{proof}
  The final equivalence follows immediately from Lemma \ref{lem:time_lag_cancellation}.  For the first equivalence we have

  \begin{align*}
    \linE{F(z)\A_{jk}(z)v_k(t)}{\H_{t - 1}^{(i)}} &= 0\ \forall \text{ strictly causal } F(z)\\
    \iff \inner{F(z)A_{jk}(z)v_k(t)}{x_i(t - \tau)} &= 0\ \forall \tau \ge 1, \text{ strictly causal } F(z),
  \end{align*}

  which can be expanded by Equation \eqref{eqn:ancestor_expansion} to
  obtain (after cancelling all ancestors of $i$ other than $k$)

  \begin{equation*}
    \inner{F(z)A_{jk}(z)v_k(t)}{\A_{ik}(z)v_k(t - \tau)} = 0\ \forall \tau \ge 1, \text{ strictly causal } F(z),
  \end{equation*}

  which by the Lemma is equivalent to $\tau_0(\A_{jk}) \ge \tau_\infty(\A_{ik})$ as stated.
\end{proof}

\begin{proposition}
  \label{prop:persistence_converse}
  Fix $i, j \in [n]$ and suppose $\exists k \in \anc{i} \cap \anc{j}$
  which confounds $i, j$.  Then, if $T_{ij}(z)$ is not causal we have
  $j \pwgc i$, and if $T_{ij}(z)$ is not anti-causal we have
  $i \pwgc j$.  Moreover, if Assumption \ref{ass:T_causality} is
  satisfied, then $j \pwgc i \iff i \pwgc j$.
\end{proposition}

\begin{proof}
  Recalling Theorem \ref{thm:granger_causality_equivalences}, consider
  some $\psi \in \H_{t - 1}^{(j)}$ and represent it as
  $\psi(t) = F(z)x_j(t)$ for some strictly causal filter $F(z)$.
  Then

  \begin{align*}
    &\inner{\psi(t) - \linE{\psi(t)}{\H_{t - 1}^{(i)}}}{x_i(t) - \linE{x_i(t)}{\H_{t - 1}^{(i)}}}\\
    &\overset{(a)}{=} \inner{F(z)x_j(t)}{x_i(t) - \linE{x_i(t)}{\H_{t - 1}^{(i)}}}\\
    &\overset{(b)}{=} \inner{F(z)\big(\A_{jj}(z)v_j(t) + \sum_{k \in \anc{j}}\A_{jk}(z)v_k(t)\big)}{(1 - H_i(z))\big(\A_{ii}(z)v_i(t) + \sum_{\ell \in \anc{i}}\A_{i\ell}(z)v_\ell(t)\big)}\\
    &\overset{(c)}{=} \sum_{k \in \anc{i}\cap\anc{j}}\inner{F(z)\A_{jk}(z)v_k(t)}{(1 - H_i(z))\A_{ik}(z)v_k(t)},
  \end{align*}

  where $(a)$ applies the orthogonality principle, $(b)$ expands with
  Equation \eqref{eqn:ancestor_expansion} with
  $H_i(z)x_i(t) = \linE{x_i(t)}{H_{t - 1}^{(i)}}$, and $(c)$ follows by
  performing cancellations of $v_k(t) \perp v_\ell(t)$ and noting that
  by the contrapositive of Proposition \ref{prop:sc_graph_common_anc}
  we cannot have $i \in \anc{j}$ or $j \in \anc{i}$.

  Through symmetric calculation, we can obtain the expression relevant
  to the determination of $i \pwgc j$ for $\phi \in \H_{t - 1}^{(i)}$ represented by the strictly causal filter $G(z): \phi(t) = G(z)x_i(t)$
  \begin{align*}
    &\inner{\phi(t) - \linE{\phi(t)}{\H_{t - 1}^{(j)}}}{x_j(t) - \linE{x_j(t)}{\H_{t - 1}^{(j)}}}\\
    &= \sum_{k \in \anc{i} \cap \anc{j}}\inner{G(z)\A_{ik}(z)v_k(t)}{(1 - H_j(z))\A_{jk}(z)v_k(t)},
  \end{align*}

  where $H_j(z)x_j(t) = \linE{x_j(t)}{\H_{t - 1}^{(j)}}.$

  We have therefore

  \begin{align}
      &(j \pwgc i): \exists F(z) \text{ s.t. } \sum_{k \in \anc{i} \cap \anc{j}}\inner{F(z)\A_{jk}(z)v_k(t)}{(1 - H_i(z))\A_{ik}(z)v_k(t)} \ne 0,\\
      &(i \pwgc j): \exists G(z) \text{ s.t. } \sum_{k \in \anc{i} \cap \anc{j}}\inner{G(z)\A_{ik}(z)v_k(t)}{(1 - H_j(z))\A_{jk}(z)v_k(t)} \ne 0.
  \end{align}

  The persistence condition, by Corollary
  \ref{cor:time_lag_cancellation}, ensures that for each
  $k \in \anc{i}\cap\anc{j}$ there is some $F(z)$ and some $G(z)$ such
  that at least one of the above terms constituting the sum over $k$
  is non-zero.  It remains to eliminate the possibility of
  cancellation in the sum.

  The adjoint of a linear filter $C(z)$ is simply $C(z^{-1})$, which
  recall is strictly anti-causal if $C(z)$ is strictly causal.  Using
  this, we can write

  \begin{align*}
    &\sum_{k \in \anc{i} \cap \anc{j}}\inner{F(z)\A_{jk}(z)v_k(t)}{(1 - H_i(z))\A_{ik}(z)v_k(t)}\\
    = &\sum_{k \in \anc{i} \cap \anc{j}}\inner{\A_{ik}(z^{-1})(1 - H_i(z^{-1}))F(z)\A_{jk}(z)v_k(t)}{v_k(t)}.\\
  \end{align*}

  Moreover, it is sufficient to find some strictly causal $F(z)$ of
  the form $F(z)(1 - H_j(z))$ (abusing notation) since $1 - H_j(z)$ is
  causal.  Similarly for $G(z)$, this leads to symmetric expressions
  for $j \pwgc i$ and $i \pwgc j$ respectively:

  \begin{equation}
    \label{eqn:T_F}
    \sum_{k \in \anc{i} \cap \anc{j}}\inner{\A_{ik}(z^{-1})(1 - H_i(z^{-1}))F(z)(1 - H_j(z))\A_{jk}(z)v_k(t)}{v_k(t)},
  \end{equation}
  \begin{equation}
    \label{eqn:T_G}
    \sum_{k \in \anc{i} \cap \anc{j}}\inner{\A_{ik}(z^{-1})(1 - H_i(z^{-1}))G(z^{-1})(1 - H_j(z))\A_{jk}(z)v_k(t)}{v_k(t)}.
  \end{equation}

  Recall the filter from Assumption \ref{ass:T_causality}

  \begin{equation}
    T_{ij}(z) = \sum_{k \in \anc{i} \cap \anc{j}} \sigma_k^2\A_{ik}(z^{-1})(1 - H_i(z^{-1}))(1 - H_j(z))\A_{jk}(z).
  \end{equation}

  Since each $v_k(t)$ is uncorrelated through time,
  $\inner{T_{ij}(z)v_k(t)}{v_k(t)} = \sigma_k^2T_{ij}(0)$, and
  therefore we have $j \pwgc i$ if $T_{ij}(z)$ is \textit{not} causal
  and $i \pwgc j$ if $T_{ij}(z)$ it \textit{not} anti-causal.
  Moreover, we have $i \npwgc j$ \textit{and} $j \pwgc i$ if
  $T_{ij}(z)$ is a constant.  Therefore, under Assumption
  \ref{ass:T_causality} $j \pwgc i \iff i \pwgc j$.

  This follows since if $T_{ij}(z)$ is not causal then
  $\exists k > 0$ such that the $z^k$ coefficient of $T_{ij}(z)$ is
  non-zero, and we can choose strictly causal $F(z) = z^{-k}$ such
  that \eqref{eqn:T_F} is non-zero and therefore $j \pwgc i$.
  
  Similarly, if $T_{ij}(z)$ is not anti-causal, then
  $\exists k > 0$ such that the $z^{-k}$ coefficient of $T_{ij}(z)$ is
  non-zero, and we can choose strictly causal $G(z)$ so that
  $G(z^{-1}) = z^k$, and then \refeq{eqn:T_G} is non-zero and therefore
  $i \pwgc j$.
\end{proof}

\subsection{The Main Theorem}
\label{apx:proof_main_theorem}

\begin{theorem}[Pairwise Recovery]
  \label{thm:scg_recovery}
  If the Granger causality graph $\gcg$ for the process $x(t)$ is a
  strongly causal DAG and Assumption \ref{ass:T_causality} holds, then
  $\gcg$ can be inferred from pairwise causality tests.  The procedure
  can be carried out, assuming we have an oracle for pairwise
  causality, via Algorithm (\ref{alg:pwgr}).
\end{theorem}

\begin{algorithm}[H]
  \SetKwInOut{Input}{input}
  \SetKwInOut{Output}{output}
  \SetKwInOut{Initialize}{initialize}
  \DontPrintSemicolon

  \caption{Pairwise Granger Causality Algorithm (PWGC)}
  \label{alg:pwgr}
  \Input{Pairwise Granger causality relations}
  \Output{Edges $\gcge = \{(i, j) \in [n] \times [n]\ |\ i \gc j \}$ of
    the graph $\gcg$.}
  \Initialize{$S_0 = [n]$  \texttt{\# unprocessed nodes}\\
    $E_0 = \emptyset$  \texttt{\# edges of }$\gcg$\\

    $k = 1$ \texttt{\# a counter used only for notation}}
  \BlankLine
  $W \leftarrow \{(i, j)\ |\ i \pwgc j, j \npwgc i\}$  \texttt{\# candidate edges}\\
  $P_0 \leftarrow \{i \in S_0\ |\ \forall s \in S_0\ (s, i) \not\in W\}$  \texttt{\# parentless nodes}\\
  \While{$S_{k - 1} \ne \emptyset$}{
    $S_k \leftarrow S_{k - 1} \setminus P_{k - 1}$ \texttt{\# remove nodes with depth }$k - 1$\\
    $P_k \leftarrow \{i \in S_k\ |\ \forall s \in S_k\ (s, i) \not\in W\}$   \texttt{\# candidate children}\\

    $D_{k0} \leftarrow \emptyset$\\
    \For{$r = 1, \ldots, k$} 
    {
      $Q \leftarrow E_{k - 1} \cup \big(\bigcup_{\ell = 0}^{r - 1} D_{k\ell}\big)$ \texttt{\# currently known edges}\\
      $D_{kr} \leftarrow \{(i, j) \in P_{k - r} \times P_k\ |\ (i, j) \in W,\ \text{no } i \rightarrow j \text{ path in } Q\}$
    }
    $E_k \leftarrow E_{k - 1} \cup \big(\bigcup_{r = 1}^k D_{kr}\big)$ \texttt{\# update } $E_k$ \texttt{ with new edges}\\
    $k \leftarrow k + 1$
  }
  \Return{$E_{k - 1}$}
\end{algorithm}

Our proof proceeds in 5 steps stated formally as lemmas.  Firstly, we
characterize the sets $W$ and $P_k$.  Then we establish a correctness
result for the inner loop on $r$, a correctness result for the outer
loop on $k$, and finally that the algorithm terminates in a finite
number of steps.

\begin{lemma}[$W$ Represents Ancestor Relations]
  \label{lem:W_subset_E}
  In Algorithm \ref{alg:pwgr} we have
  $(i, j) \in W$ if and only if $i \in \anc{j}$.  In particular,
  $W \subseteq \gcge$.
\end{lemma}
\begin{proof}
  Let $j \in [n]$ and suppose that $i \in \anc{j}$.  Then $i \pwgc j$
  by Proposition \ref{prop:pwgc_anc}.  Proposition
  \ref{prop:sc_graph_common_anc} ensures that $(i, j)$ are not
  confounded and Corollary \ref{cor:parent_corollary} that
  $j \not\in \anc{i}$ so $j \npwgc i$ by Proposition and therefore
  \ref{prop:ancestor_properties} $(i, j) \in W$.

  Conversely, suppose $(i, j) \in W$.  Then since $j \npwgc i$,
  Proposition \ref{prop:persistence_converse} ensures that $(j, i)$
  are not confounded and so by Proposition \ref{prop:ancestor_properties}
  we must have $i \in \anc{j}$.
\end{proof}

\begin{definition}[Depth]
  For our present purposes we will define the \textit{depth} $d(j)$ of
  a node $j$ in $\gcg$ to be the length of the \textit{longest} path
  from a node in $P_0$ to $j$, where $d(j) = 0$ if $j \in P_0$.  It is
  apparent that such a path will always exist.  For example, in Figure
  \ref{fig:example_fig3} we have $d(3) = 1$ and $d(4) = 2$.
\end{definition}

\begin{lemma}[Depth Characterization of $P_k$ and $S_k$]
  \label{lem:depth_lemma}
  $i \in P_k \iff d(i) = k$ and $j \in S_k \iff d(j) \ge k$.
\end{lemma}
\begin{proof}
  We proceed by induction, noting that $P_0$ is non-empty since $\gcg$
  is acyclic and therefore $\gcg$ contains nodes without parents.  The
  base case $i \in P_0 \iff d(i) = 0$ is by definition, and
  $j \in S_0 \iff d(j) \ge 0$ is trivial since $S_0 = [n]$.  So
  suppose that the lemma is true up to $k - 1$.

  ($i \in P_k \implies d(i) = k$): Let $i \in P_k$.  Suppose that
  $d(i) \ge k + 1$, then $\exists j \in \pa{i}$ such that
  $j \not\in \cup_{r \ge 1}P_{k - r}$ (otherwise $d(i) \le k$), this
  implies that $j \in S_k$ with $(j, i) \in W$ (by Lemma
  \ref{lem:W_subset_E}) which is not possible due to the construction of
  $P_k$ and therefore $d(i) \le k$.  Moreover,
  $P_k \subseteq S_k \subseteq S_{k - 1}$ implies that
  $d(i) \ge k - 1$ by the induction hypothesis, but if $d(i) = k - 1$
  then $i \in P_{k - 1}$ again by induction which is impossible since
  $i \in P_k$ and therefore $d(i) = k$.

  ($s \in S_k \implies d(s) \ge k$): Let
  $s \in S_k \subseteq S_{k - 1}$.  We have by induction that
  $d(s) \ge k - 1$, but again by induction (this time on $P_{k - 1}$)
  we have $d(s) \ne k - 1$ since $S_k = S_{k - 1} \setminus P_{k - 1}$
  and therefore $d(s) \ge k$.

  ($d(i) = k \implies i \in P_k$): Suppose $i \in [n]$ is such that
  $d(i) = k$.  Then $i \in S_{k - 1}$ by the hypothesis, but also
  $i \not\in P_{k - 1}$ so then
  $i \in S_k = S_{k - 1} \setminus P_{k - 1}$.  Now, recalling the
  definition of $P_k$

  \begin{equation*}
    P_k = \{i \in S_k\ |\ \forall s \in S_k\ (s, i) \not\in W \},
  \end{equation*}

  if $s \in S_k$ is such that $(s, i) \in W$ then $s \pwgc i$ and
  $i \npwgc s$ so that by Proposition \ref{prop:persistence_converse}
  there cannot be a confounder of $(s, i)$ (otherwise $i \pwgc s$) so
  then by Proposition \ref{prop:ancestor_properties} we have
  $s \in \anc{i}$.  We have shown that $s \in S_k \implies d(s) \ge k$
  and so we must have $d(i) > k$, a contradiction, therefore there is
  no such $s \in S_k$ so $i \in P_k$.

  ($d(j) \ge k \implies j \in S_k$): Let $j \in [n]$ such that
  $d(j) \ge k$, then by induction we have $j \in S_{k - 1}$.  This
  implies by the construction of $S_k$ that $j \not\in S_k$ only if
  $j \in P_{k - 1}$, but we have shown that this only occurs when
  $d(j) = k - 1$, but $d(j) > k - 1$ so $j \in S_k$.
\end{proof}

\begin{lemma}[Inner Loop]
  \label{lem:inner_loop_lemma}
  Fix an integer $k \ge 1$ and suppose that $(i, j) \in E_{k - 1}$ if
  and only if $(i, j) \in \gcge$ and $d(j) \le k - 1$.  Then, we have
  $(i, j) \in D_{kr}$ if and only if $(i, j) \in \gcge $, $d(j) = k$,
  and $d(i) = k - r$.
\end{lemma}
\begin{proof}
  We prove by induction on $r$, keeping in mind the results of Lemmas
  \ref{lem:W_subset_E} and \ref{lem:depth_lemma}.  For the base case,
  let $r = 1$ and suppose that $(i, j) \in \gcge$ with $d(j) = k$ and
  $d(i) = k - 1$.  Then, by Corollary \ref{cor:gc_implies_pwgc}
  $(i, j) \in W$ and by our assumptions on $E_{k - 1}$ there is no
  $\gcgpath{i}{j}$ path in $E_{k - 1}$ and therefore
  $(i, j) \in D_{k1}$.  Conversely, suppose that $(i, j) \in D_{k1}$.
  Then, $d(i) = k - 1$ and $d(j) = k$ which, since
  $(i, j) \in W \implies i \in \anc{j}$ implies that $i \in \pa{j}$
  and $(i, j) \in \gcge$.

  Now, fix $r > 1$ and suppose that the result holds up to $r - 1$.
  Let $(i, j) \in \gcge$ with $d(j) = k$ and $d(i) = k - r$.  Then,
  $(i, j) \in W$ and by induction and strong causality there cannot
  already be an $\gcgpath{i}{j}$ path in
  $E_{k - 1} \cup \big(\bigcup_{\ell = 0}^{r - 1} D_{kr}\big)$,
  therefore $(i, j) \in D_{kr}$.  Conversely, suppose
  $(i, j) \in D_{kr}$.  Then we have $d(i) = k - r$, $d(j) = k$, and
  $i \in \anc{j}$.  Suppose by way of contradiction that
  $i \not\in \pa{j}$, then there must be some $u \in \pa{j}$ such that
  $i \in \anc{u}$.  But, this implies that $d(i) < d(u)$ and by
  induction that $(u, j) \in \bigcup_{\ell = 1}^{r - 1}D_{k\ell}$.
  Moreover, since $d(u) < k$ (otherwise $d(j) > k$) each edge in
  the $\gcgpath{i}{u}$ path must already be in $E_{k - 1}$, and so
  there must be an $\gcgpath{i}{j}$ path in
  $E_{k - 1}\cup\big(\bigcup_{\ell = 0}^{r - 1}D_{kr}\big)$, which is
  a contradiction since we assumed $(i, j) \in D_{kr}$.  Therefore
  $i \in \pa{j}$ and $(i, j) \in \gcge$.
\end{proof}

\begin{lemma}[Outer Loop]
  \label{lem:outer_loop_lemma}
  We have $(i, j) \in E_k$ if and only if $(i, j) \in \gcge$ and
  $d(j) \le k$.  That is, at iteration $k, E_k$ and $\gcge$ agree on
  the set of edges whose terminating node is at most $k$ steps away
  from $P_0$.
\end{lemma}
\begin{proof}
  We will proceed by induction.  The base case $E_0 = \emptyset$ is
  trivial, so fix some $k \ge 1$, and suppose that the lemma holds for
  all nodes of depth less than $k$.

  Suppose that
  $(i, j) \in E_k = E_{k - 1}\cup \big(\bigcup_{r = 1}^k D_{rk}
  \big)$.  Then clearly there is some $1 \le r \le k$ such that
  $(i, j) \in D_{kr}$ so that by Lemma \ref{lem:inner_loop_lemma} we
  have $(i, j) \in \gcge$ and $d(j) = k$.

  Conversely, suppose that $(i, j) \in \gcge$ and $d(j) \le k$.  If
  $d(j) < k$ then by induction $(i, j) \in E_{k - 1} \subseteq E_k$ so
  suppose further than $d(j) = k$.  Since $i \in \pa{j}$ we must have
  $d(i) < k$ (else $d(j) > k$) and again by Lemma
  \ref{lem:inner_loop_lemma} $(i, j) \in \bigcup_{r = 1}^k D_{kr}$
  which implies that $(i, j) \in E_k$.
\end{proof}

\begin{lemma}[Finite Termination]
  Algorithm \ref{alg:pwgr} terminates and returns the set
  $E_{k^\star - 1} = \gcge$ for some $k^\star \le n$.
\end{lemma}
\begin{proof}
  If $n = 1$, the algorithm is clearly correct, returning on the first
  iteration with $E_1 = \emptyset$.  When $n > 1$ Lemma
  \ref{lem:outer_loop_lemma} ensures that $E_k$ coincides with
  $\{(i, j) \in \gcge\ |\ d(j) \le k\}$ and since $d(j) \le n - 1$ for
  any $j \in [n]$ there is some $k^\star \le n$ such that
  $E_{k^\star - 1} = \gcge$.  We must have $S_{k^\star} = \emptyset$
  since $j \in S_{k^\star} \iff d(j) \ge k^\star$ (if $d(j) > k - 1$ then
  $E_{k^\star - 1} \ne \gcge$) and therefore the algorithm terminates.
\end{proof}

\section{Finite Sample Implementation}
\label{sec:structure_learning}
In this section we provide a review of our methods for implementing
Algorithm 1 given a \textit{finite} sample of $T$ data points.  We
apply the simplest reasonable methods in order to maintain a focus on
our main contributions (i.e. Algorithm \ref{alg:pwgr}), more
sophisticated schemes can only serve to improve the results.  Textbook
reviews of the following concepts are provided e.g. by
\cite{all_of_statistics}, \cite{murphy_mlp}, and elsewhere.

In subsection \ref{sec:pairwise_hypothesis_testing} we define pairwise
Granger causality hypothesis tests, in subsection
\ref{sec:model_order_selection} a model order selection criteria, in
subsection \ref{sec:efficient_model_estimation} an efficient
estimation algorithm, in subsection \ref{sec:error_rate_control} the
method for choosing an hypothesis testing threshold, and finally in
subsection \ref{sec:finite_pwgc} the unified finite sample algorithm.

\subsection{Pairwise Hypothesis Testing}
\label{sec:pairwise_hypothesis_testing}
In performing pairwise checks for Granger causality $x_j \pwgc x_i$ we
follow the simple scheme of estimating the following two linear models:

\begin{align}
  H_0:&\ \widehat{x}_i^{(p)}(t) = \sum_{\tau = 1}^{p} b_{ii}(\tau)x_i(t - \tau),\\
  H_1:&\ \widehat{x}_{i|j}^{(p)}(t) = \sum_{\tau = 1}^{p} b_{ii}(\tau)x_i(t - \tau) + \sum_{\tau = 1}^pb_{ij}(\tau)x_j(t - \tau).
\end{align}

We formulate the statistic 

\begin{equation}
  \label{eqn:gc_statistics}
  F_{ij}(p) = \frac{T}{p}\Big(\frac{\xi_i(p)}{\xi_{ij}(p)} - 1\Big),
\end{equation}

where $\xi_i(p)$ is the sample mean square of the
residuals\footnote{This quantity is often denoted $\widehat{\sigma}$,
  but we maintain notation from Definition
  \ref{def:granger_causality}.}  $x_i(t) - \widehat{x}^{(p)}_i(t)$,

\begin{equation*}
  \xi_i(p) = \frac{1}{T - p}\sum_{t = p + 1}^T (x_i(t) - \widehat{x}_i^{(p)}(t))^2,
\end{equation*}

and similarly for $\xi_{ij}(p)$.  We test $F_{ij}(p)$ against a
$\chi^2(p)$ distribution.

If the estimation procedure is consistent, we will have the following
convergence (in $\P$ or a.s.):

\begin{equation}
  F_{ij}(p) \rightarrow
  \left\{
    \begin{array}{ll}
      0;\ x_j \npwgc x_i\\
      \infty;\ x_j \pwgc x_i
    \end{array}
  \right. \text{ as } T \rightarrow \infty.  
\end{equation}

In our finite sample implementation (see Algorithm
\ref{alg:finite_pwgc}) we add edges to $\widehat{\gcg}$ in order of
the decreasing magnitude of $F_{ij}$ instead of proceeding backwards
through $P_{k - r}$ in Algorithm \ref{alg:pwgr}.  This makes greater
use of the information provided by the test statistic $F_{ij}$,
moreover, if $x_i \gc x_j$ and $x_j \gc x_k$, it is expected that
$F_{kj} > F_{ki}$, thereby providing the same effect as proceeding
backwards through $P_{k - r}$.
\subsection{Model Order Selection}
\label{sec:model_order_selection}
There are a variety of methods to choose the filter order $p$ (see
e.g. \cite{lutkepohl2005new}), but we will focus in particular on the
Bayesian Information Criteria (BIC).  The BIC is substantially more
conservative than the popular alternative Akaiake Information Criteria
(the BIC is also asymptotically consistent), and since we are
searching for \textit{sparse graphs}, we therefore prefer the BIC,
where we seek to \textit{minimize} over $p$:

\begin{equation}
  \label{eqn:bic}
  \begin{aligned}
    BIC_{\text{univariate}}(p) &= \ln\ \xi_i(p) + p\frac{\ln T}{T},\\
    BIC_{\text{bivariate}}(p) &= \ln \det \widehat{\Sigma}_{ij}(p) + 4p\frac{\ln T}{T},\\
  \end{aligned}
\end{equation}

where $\widehat{\Sigma}_{ij}(p)$ is the $2 \times 2$ residual
covariance matrix for the $\VAR(p)$ model of $(x_i(t), x_j(t))$.  The
bivariate errors $\xi_{ij}(p)$ and $\xi_{ji}(p)$ are the diagonal
entries of $\widehat{\Sigma}_{ij}(p)$.

We carry this out by a simple direct search on each model order
between $0$ and some prescribed $p_\text{max}$, resulting in a
collection $p_{ij}$ of model order estimates.  In practice, it is
sufficient to pick $p_\text{max}$ ad-hoc or via some simple heuristic
e.g. plotting the sequence $BIC(p)$ over $p$, though it is not
technically possible to guarantee that the optimal $p$ is less than
the chosen $p_\text{max}$ (since there can in general be arbitrarily
long lags from one variable to another).

\subsection{Efficient Model Estimation}
\label{sec:efficient_model_estimation}
In practice, the vast majority of computational effort involved in
implementing our estimation algorithm is spent calculating the error
estimates $\xi_i(p_i)$ and $\xi_{ij}(p_{ij})$.  This requires fitting a
total of $n^2p_{\text{max}}$ autoregressive models, where the most
naive algorithm (e.g. solving a least squares problem for each model)
for this task will consume $O(n^2p_{\text{max}}^4T)$ time, it is
possible to carry out this task in a much more modest
$O(n^2p_{\text{max}}^2 ) + O(n^2p_{\text{max}}T)$ time via the
autocorrelation method
\cite{hayes_statistical_digital_signal_processing} which substitutes
the following autocovariance estimates in the Yule-Walker
equations:\footnote{The particular indexing and normalization given in
  Equation \eqref{eqn:covariance_estimate} is critical to ensure
  $\widehat{R}$ is positive semidefinite.  The estimate can be viewed
  as calculating the covariance sequence of a signal multiplied by a
  rectangular window.}

\begin{equation}
  \label{eqn:covariance_estimate}
  \widehat{R}_x(\tau) = \frac{1}{T}\sum_{t = \tau + 1}^T x(t) x(t - \tau)^\T;\ \tau = 0, \ldots, p_{\text{max}},
\end{equation}

It is imperative that the first index in the summation is $\tau + 1$, as
opposed perhaps to $p_\text{max}$ and that the normalization is
$1 / T$, as opposed perhaps to $1 / (T - p_\text{max})$, in order to
guarantee that $\widehat{R}_x(\tau)$ forms a valid (i.e. positive
definite) covariance sequence.  This results in some bias, however the
dramatic computational speedup is worth it for our purposes.

These covariance estimates constitute the $O(n^2p_{\text{max}}T)$
operation.  Given these particular estimates, the variances $\xi_i(p)$
for $p = 1, \ldots, p_{\text{max}}$ can be evaluated in
$O(p_{\text{max}}^2)$ time each by applying the Levinson-Durbin
recursion to $\widehat{R}_{ii}(\tau)$, which effectively estimates a
sequence of $AR$ models, producing $\xi_i(p)$ as a side-effect (see
\cite{hayes_statistical_digital_signal_processing} and
\cite{levinson_durbin_recursion}).

Similarly, the variance estimates $\widehat{\Sigma}_{ij}(p)$ (which
include $\xi_{ij}$ and $\xi_{ji}$) can be obtained by estimating
$\frac{(n + 1)n}{2}$ bivariate AR models, again in
$O(p_{\text{max}}^2)$ time via Whittle's generalized Levinson-Durbin
recursion\footnote{We have made use of standalone tailor made
  implementations of these algorithms, available at
  \textsf{github.com/RJTK/Levinson-Durbin-Recursion}.}
\cite{whittle_generalized_levinson_durbin}.

\subsection{Edge Probabilities and Error Rate Controls}
\label{sec:error_rate_control}
Denote $F_{ij}$ the Granger causality statistic of Equation
\eqref{eqn:gc_statistics} with model orders chosen by the methods of
Section \ref{sec:model_order_selection}.  We assume that this
statistic is asymptotically $\chi^2(p_{ij})$ distributed (e.g. the
disturbances are Gaussian), and denote by $G$ the cumulative
distribution function thereof.  We will define the matrix

\begin{equation}
  \label{eqn:edge_inclusion_probability}
  P_{ij} = G(F_{ij}),
\end{equation}

to be the matrix of pairwise edge inclusion P-values.  This is
motivated by the hypothesis test where the hypothesis $H_0$ will be
rejected (and thence we will conclude that $x_j \pwgc x_i$) if
$P_{ij} > 1 - \delta$.

The value $\delta$ can be chosen by a variety of methods, in our case
we apply the Benjamini Hochberg criteria \cite{benjamini_hochberg}
\cite{all_of_statistics} to control the false discovery rate of
pairwise edges to a level $\alpha$ (where we generally take
$\alpha = 0.05$).

\subsection{Finite Sample Recovery Algorithm}
\label{sec:finite_pwgc}

After the graph topology $\widehat{\gcg}$ has been estimated via
Algorithm \ref{alg:finite_pwgc}, we refit the entire model with the
specified sparsity pattern directly via ordinary least squares.

We note that producing graph estimates which are not strongly causal
can potentially be achieved by performing sequential estimates
$\widehat{x}_1(t), \widehat{x}_2(t), \ldots$ estimating a strongly causal
graph with the residuals of the previous model as input, and then
refitting on the combined sparsity pattern.  We intend to consider
this heuristic in future work.

\begin{algorithm}
    \SetKwInOut{Input}{input}
    \SetKwInOut{Output}{output}
    \SetKwInOut{Initialize}{initialize}
    \DontPrintSemicolon

    \BlankLine
    \caption{Finite Sample Pairwise Graph Recovery (PWGC)}
    \label{alg:finite_pwgc}

    \Input{Estimates of pairwise Granger causality statistics $F_{ij}$
      (eqn. \ref{eqn:gc_statistics}).  Matrix of edge probabilities $P_{ij}$ (eqn. \ref{eqn:edge_inclusion_probability}).  Hypothesis testing threshold $\delta$ chosen via the Benjamini-Hochberg criterion (Section \ref{sec:error_rate_control})}
    \Output{A strongly causal graph $\widehat{\gcg}$}
    \Initialize{$S = [n]$  \texttt{\# unprocessed nodes}\\
      $E = \emptyset$  \texttt{\# edges of }$\widehat{\gcg}$\\
      $k = 1$ \texttt{\# a counter used only for notation}}

    \BlankLine

    $W_\delta \leftarrow \{(i, j)\ |\ P_{ji} > 1 - \delta, F_{ji} > F_{ij}\}$  \texttt{\# candidate edges}\\
    $\mathcal{I}_0 \leftarrow \big(\sum_{j\in S: (j, i) \in W_\delta} P_{ij}, \mathsf{\ for\ }i \in S \big)$ \texttt{\# total node incident probability}\\
    $P_0 \leftarrow \{i \in S\ |\ \mathcal{I}_0(i) < \ceil{\text{min}(\mathcal{I}_0)}\}$ \texttt{\# Nodes with fewest incident edges}\\
    \If{$P_0 = \emptyset$}{
      $P_0 \leftarrow \{i \in S\ |\ \mathcal{I}_0(i) \le \ceil{\text{min}(\mathcal{I}_0)}\}$ \texttt{\# Ensure non-empty}
    }
    \BlankLine

    \While{$S \ne \emptyset$}{
      $S \leftarrow S \setminus P_{k - 1}$ \texttt{\# remove processed nodes}\\
      $\mathcal{I}_k \leftarrow \big(\sum_{j\in S: (j, i) \in W_\delta} P_{ij}, \mathsf{\ for\ }i \in S \big)$\\
      $P_k \leftarrow \{i \in S\ |\ \mathcal{I}_k(i) < \ceil{\text{min}(\mathcal{I}_k)}\}$\\
      \If{$P_k = \emptyset$}{
        $P_k \leftarrow \{i \in S\ |\ \mathcal{I}_k(i) \le \ceil{\text{min}(\mathcal{I}_k)}\}$
      }
      \;
      \texttt{\# add strongest edges, maintaining strong causality}\\
      $U_k \leftarrow \bigcup_{r = 1}^k P_{k - r}$ \texttt{\# Include all forward edges}\\
      \For{$(i, j) \in \mathsf{sort}\Big(\{(i, j) \in U_k \times P_k\ |\ (i, j) \in W_\delta\} \mathsf{\ by\ descending\ } F_{ji}\Big)$} {
        \If{$\mathsf{is\_strongly\_causal}(E \cup \{(i, j)\})$} {
          \texttt{\# }$\mathsf{is\_strongly\_causal}$ \texttt{can be implemented by keeping}\\
          \texttt{\# track of ancestor / descendant relationships}\\
          $E \leftarrow E \cup \{(i, j)\}$
        }
      }
      $k \leftarrow k + 1$\\
    }
    \Return{$([n], E)$}
\end{algorithm}

\section{Simulation}
\label{apx:simulation}
We have implemented our empirical experiments in Python \cite{scipy},
in particular we leverage the LASSO implementation from
\texttt{sklearn} \cite{sklearn} and the random graph generators from
\texttt{networkx} \cite{networkx}.  We run experiments using two
separate graph topologies having $n = 50$ nodes. These are generated
respectively by drawing a random tree and a random Erdos Renyi graph
then creating a directed graph by directing edges from lower numbered
nodes to higher numbered nodes.

We populate each of the edges (including self loops) with random
linear filters constructed by placing $5$ transfer function poles
(i.e. $p = 5$) uniformly at random in a disc of radius $3 / 4$ (which
guarantees stability for acyclic graphs).  The resulting system is
driven by i.i.d. Gaussian random noise, each component having random
variance $\sigma_i^2 = 1/2 + r_i$ where $r_i \sim \text{exp}(1/2)$.  To ensure
we are generating data from a stationary system, we first discard
samples during a long burnin period.

For both PWGC and adaLASSO We set the maximum lag length
$p_{\text{max}} = 10$.

Results are collected in Figures
\ref{fig:simulation_results_comparison1},
\ref{fig:simulation_results_comparison2},
\ref{fig:simulation_results_scaling_and_small_T},
\ref{fig:simulation_results_dense}.

\begin{figure}
  \centering
  \caption{PWGC Compared Against AdaLASSO \cite{adaptive_lasso_zou2006} (SCG)}
  \label{fig:simulation_results_comparison1}
  \includegraphics[width=0.95\textwidth]{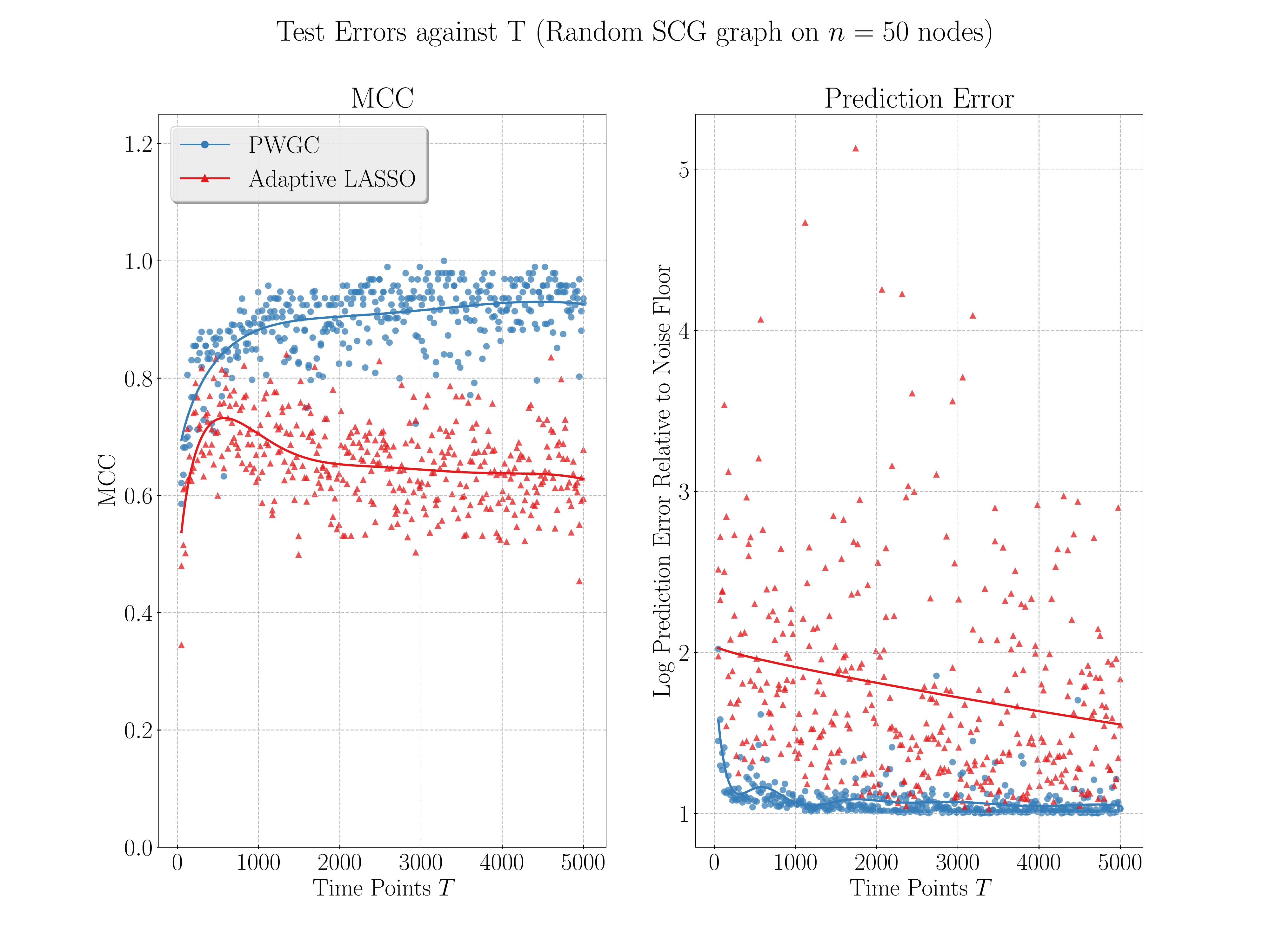}

  {\footnotesize Comparison of PWGC and LASSO for $\VAR(p)$ model
    estimation.  We make comparisons against both the MCC and the
    relative log mean-squared prediction error
    $\frac{\ln\tr \widehat{\Sigma}_v}{\ln\tr \Sigma_v}$.  Results
    in Figure \ref{fig:simulation_results_comparison1} are for systems
    guaranteed to satisfy the assumptions required for Theorem
    \ref{thm:scg_recovery}.}
\end{figure}

\begin{figure}
  \caption{PWGC vs adaLASSO (DAG, $q = \frac{2}{n}$)}
  \label{fig:simulation_results_comparison2}
  \includegraphics[width=0.95\textwidth]{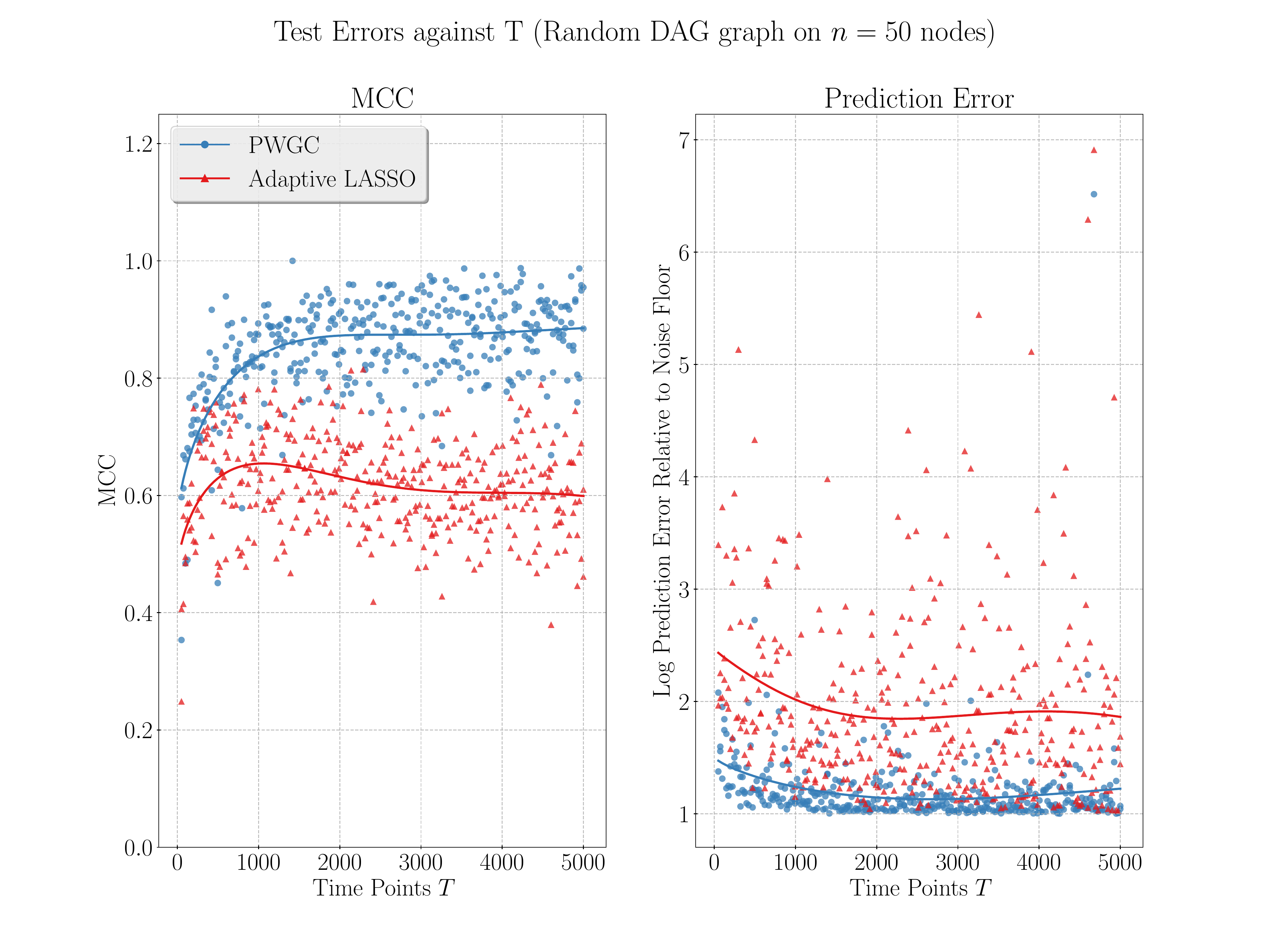}

  {\footnotesize Figure \ref{fig:simulation_results_comparison2}
    provides results for systems which do not guarantee the
    assumptions of Theorem \ref{thm:scg_recovery}, though the graph
    has a similar level of sparsity.}
\end{figure}

\begin{figure}
  \centering
  \caption{PWGC Scaling and Small Sample Performance}
  \label{fig:simulation_results_scaling_and_small_T}
  \begin{subfigure}[b]{0.45\textwidth}
    \caption{Fixed $T$, increasing $n$ (SCG)}
    \label{fig:simulation_results_scaling}
    \includegraphics[width=\linewidth]{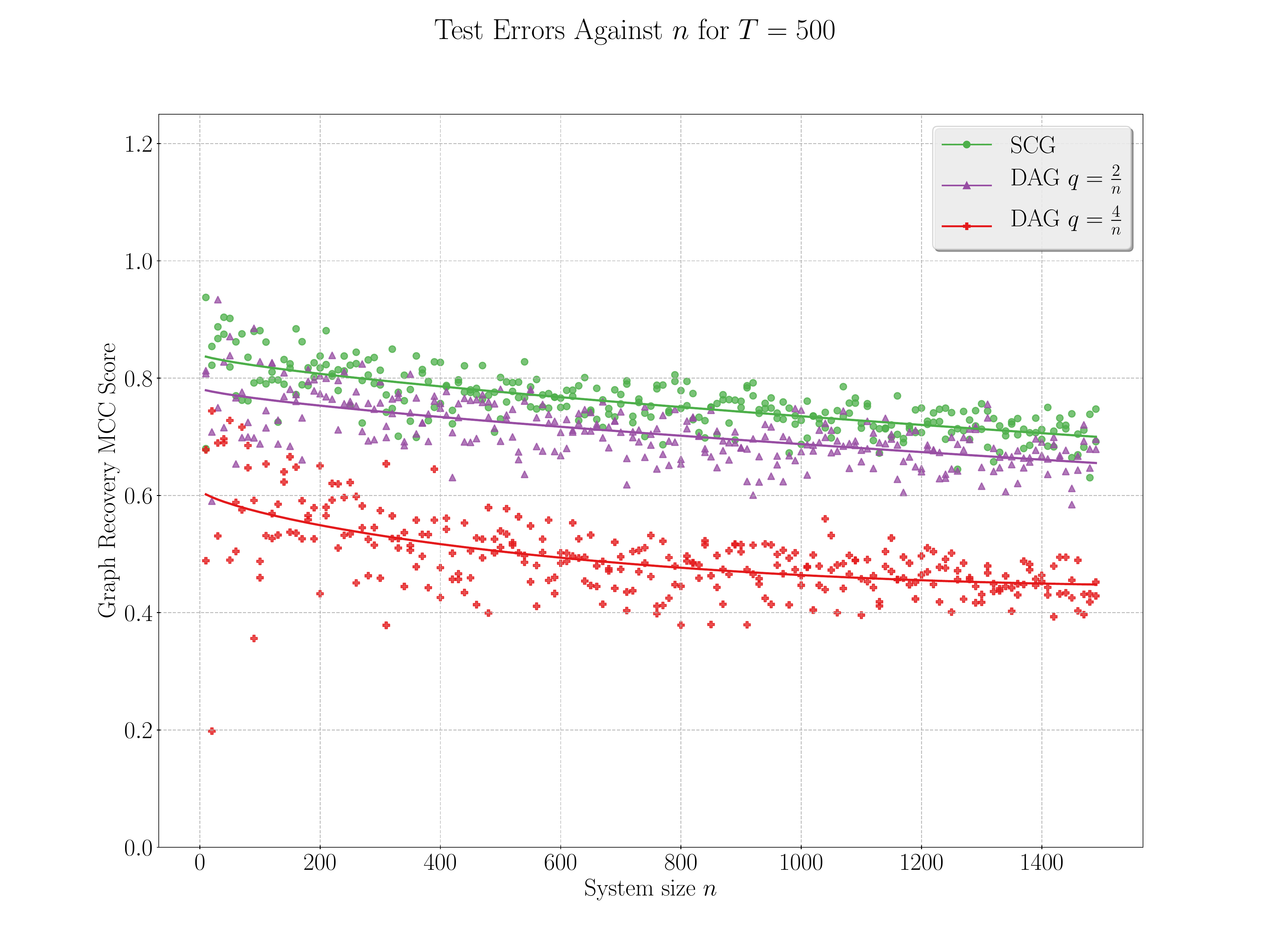}
  \end{subfigure}
  \begin{subfigure}[b]{0.45\textwidth}
    \caption{MCC Comparison for $T \le 100$}
    \label{fig:small_T_comparison}
    \includegraphics[width=\linewidth]{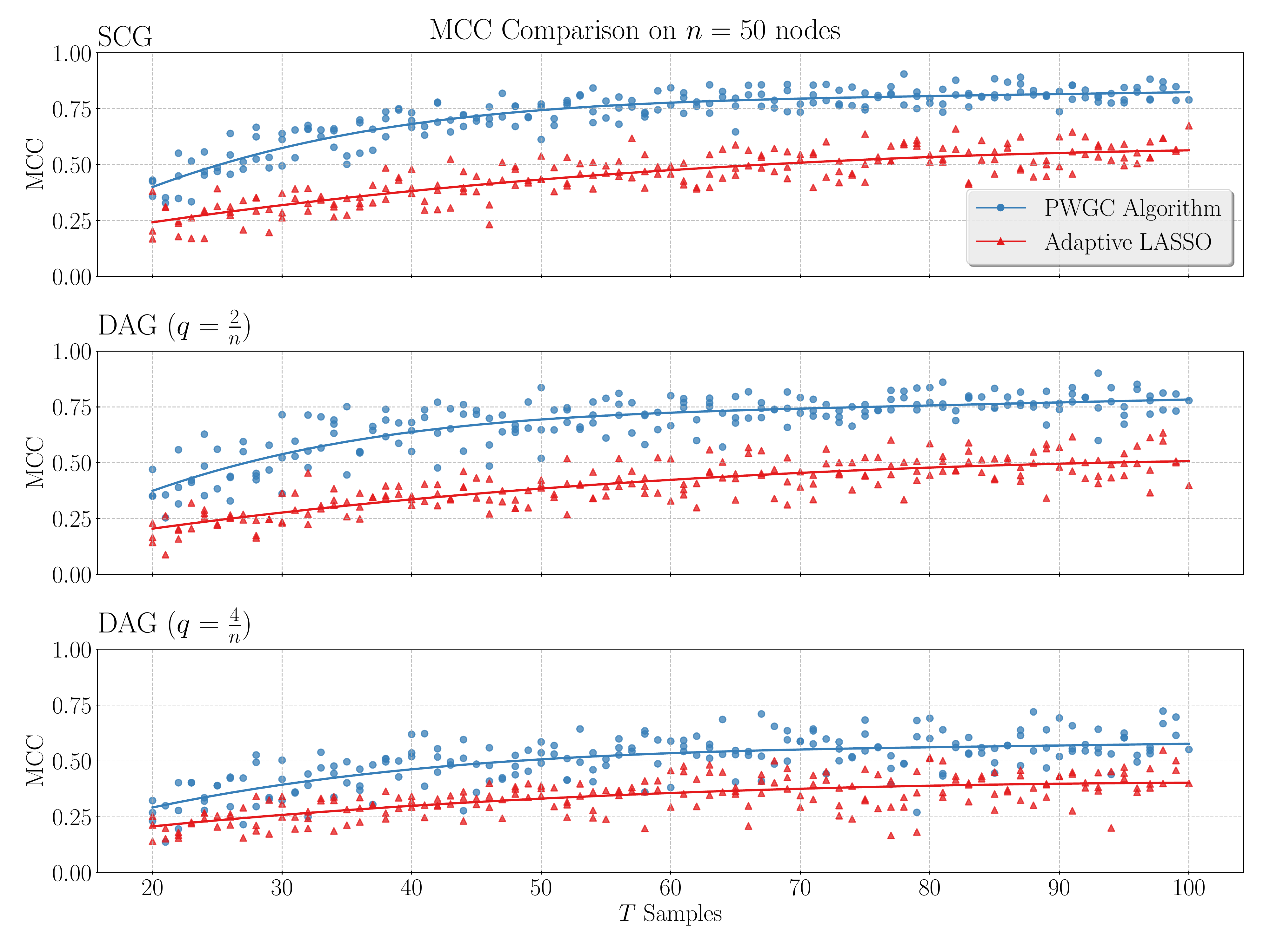}
  \end{subfigure}

  {\footnotesize Figure \ref{fig:simulation_results_scaling} measures
    support recovery performance as the number of nodes $n$ increases,
    and the edge proportion as well as the number of samples $T$ is
    held fixed.  Remarkably, the degradation as $n$ increases is
    limited, it is primarily the graph topology (SCG or non-SCG) as
    well as the level of sparsity (measured by $q$) which are the
    determining factors for support recovery performance.

    Figure \ref{fig:small_T_comparison} provides a support recovery
    comparison for very small values of, $T$ typical for many
    applications.}
\end{figure}

\begin{figure}
  \caption{Fixed $T, n$, increasing edges $q$ (DAG)}
  \label{fig:simulation_results_dense}
  \includegraphics[width=\linewidth]{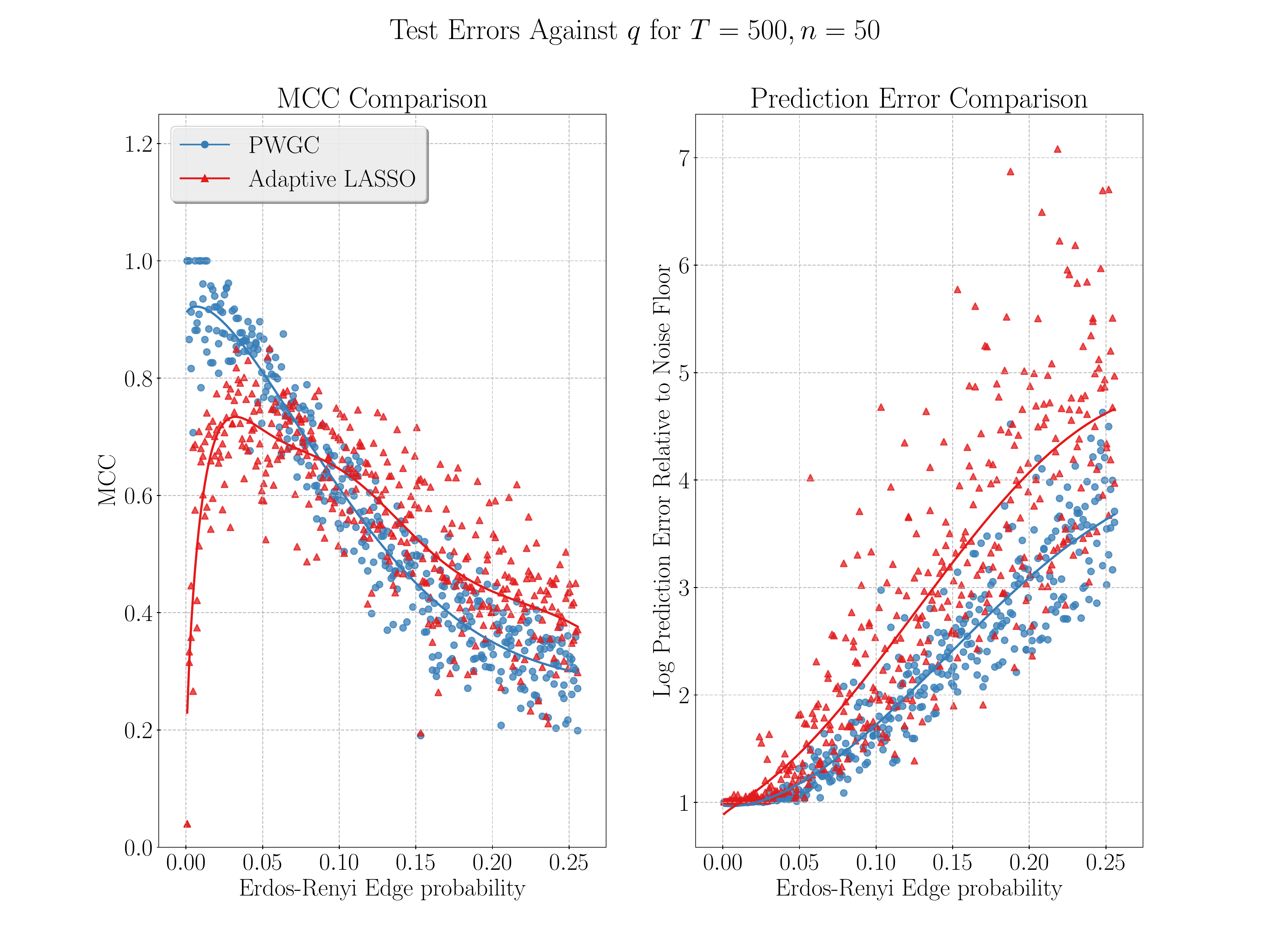}

  {\footnotesize Figure \ref{fig:simulation_results_dense} provides a
    comparison between PWGC and AdaLASSO as the density of graph edges
    (as measured by $q$) increases.  For reference,
    $\frac{2}{n} = 0.04$ has approximately the same level of sparsity
    as the SCGs we simulated.  As $q$ increases, the AdaLASSO
    outperforms PWGC as measured by the MCC.  However, PWGC maintains
    superior performance for 1-step-ahead prediction.  We speculate
    that this is a result of fitting the sparsity pattern recovered by
    PWGC via OLS which directly seeks to optimize this metric, whereas
    the LASSO is encumbered by the sparsity inducing penalty.}
\end{figure}

In reference to Figure \ref{fig:simulation_results_comparison1} it
should not be overly surprising that our PWGC algorithm performs
better than the LASSO for the case of a strongly causal graph, since
in this case the theory from which our heuristic derives is valid.
However, the performance is still markedly superior in the case of a
more general DAG.  We would conjecture that a DAG having a similar
degree of sparsity as an SCG is ``likely'' to be ``close'' to an SCG,
in some appropriate sense.

Figure \ref{fig:simulation_results_dense} illustrates the severe
(expected) degradation in performance as the number of edges increases
while the number of data samples $T$ remains fixed.  For larger values
$q$ in this plot, the number of edges in the graph is comparable to
the number of data samples.

We have also paid close attention to the performance of PWGC in the
very small sample ($T \le 100$) regime (see Figure
\ref{fig:small_T_comparison}), as this is the regime many applications
must contend with.

In regards scalability, we have observed that performing the $O(n^2)$
pairwise Granger causality calculations consumes the vast majority
($> 90\%$) of the computation time.  Since this step is trivially
parallelizable, our algorithm also scales well with multiple cores or
multiple machines.  Figure \ref{fig:simulation_results_scaling} is a
demonstration of this scalability, where we are able to estimate
graphs having over $1500$ nodes (over $2.25 \times 10 ^6$ possible edges)
using only $T = 500$ data points, granted, an SCG on this many nodes
is extremely sparse.

\clearpage
\end{appendices}

\printbibliography
\end{document}